\documentclass[12pt,twoside,leqno]{article}
\usepackage[T1]{fontenc}
\usepackage{amsfonts}
\usepackage{amsmath,amsthm}
\usepackage{amssymb,latexsym}
\usepackage{enumerate}

\theoremstyle{plain}
\newtheorem{theorem}{Theorem}[section]
\newtheorem{lemma}[theorem]{Lemma}
\newtheorem{proposition}[theorem]{Proposition}
\newtheorem{corollary}[theorem]{Corollary}

\theoremstyle{definition}

\newtheorem{definition}[theorem]{Definition}
\theoremstyle{remark}
\newtheorem{remark}[theorem]{Remark}
\newtheorem{question}[theorem]{Question}

\begin{document}
\title{Cuf products and cuf sums of (quasi)-metrizable spaces in $\mathbf{ZF}$}
\author{Kyriakos Keremedis and Eliza Wajch\\
Department of Mathematics, University of the Aegean\\
Karlovassi, Samos 83200, Greece\\
kker@aegean.gr\\
Institute of Mathematics\\
Faculty of Exact and Natural Sciences \\
Siedlce University of Natural Sciences and Humanities\\
ul. 3 Maja 54, 08-110 Siedlce, Poland\\
eliza.wajch@wp.pl}
\maketitle
\begin{abstract}
A cuf space (set, resp.) is a space (set, resp.) which is a countable union of finite subspaces (subsets, resp.). It is proved in $\mathbf{ZF}$ (with the absence of the axiom of choice) that all countable unions of cuf (denumerable, resp.) sets are cuf sets iff all countable products of cofinite cuf (denumerable, resp.) spaces are quasi-metrizable iff all countable products of one-point Hausdorff compactifications of infinite cuf (denumerable, resp.) spaces are quasi-metrizable.  A countable product of one-point Hausdorff compactifications of denumerable discrete spaces is first-countable iff it is quasi-metrizable. A model of $\mathbf{ZF}$ is shown in which a countable product two-point Hausdorff compactifications of denumerable discrete spaces is first-countable without being quasi-metrizable. Other relevant independence results are also proved.\medskip

\noindent\textit{Mathematics Subject Classification (2010)}: 03E25, 03E35, 54A35,\newline 54D35, 54E35 \newline 
\textit{Keywords}: Weak forms of the Axiom of Choice, cuf set, countable product, quasi-metrizability, cofinite cuf space, one-point Hausdorff compactification, two-point Hausdorff compactification, $\mathbf{ZF}$
\end{abstract}

\section{Preliminaries}
\label{s1}
In this paper, the intended context for reasoning and statements of theorems
is the Zermelo-Fraenkel set theory $\mathbf{ZF}$ without the axiom of choice 
$\mathbf{AC}$. In this section, we establish notation and terminology.  Section 2 is the introduction to the content of the paper. The main results of the paper are included in Section 3. Section 4 contains a list of open problems and comments on them. 

\subsection{A little more on the set-theoretic framework}
In this paper, the intended context for reasoning and statements of theorems
is the Zermelo-Fraenkel set theory $\mathbf{ZF}$ without the axiom of choice 
$\mathbf{AC}$. In this section, we establish notation and terminology.  Section 2 is the introduction to the content of the paper. The main results of the paper are included in Section 3. Section 4 contains a list of open problems. 

 The system $\mathbf{ZF+AC}$ is denoted by $\mathbf{ZFC}$. We recommend  \cite{ku1} and \cite{Ku} as a
good introduction to $\mathbf{ZF}$. To stress the fact that a result is proved in $\mathbf{ZF}$ or $\mathbf{ZF+A}$ (where $\mathbf{A}$ is a statement independent of $\mathbf{ZF}$), we shall write
at the beginning of the statements of the theorems and propositions ($%
\mathbf{ZF}$) or ($\mathbf{ZF+A}$), respectively. Apart from models of $\mathbf{ZF}$, we refer to some models
of $\mathbf{ZFA}$ (or $\text{ZF}^0$ in \cite{hr}), that is, we refer also to $\mathbf{ZF}$ with an infinite set of atoms (see \cite{j}, \cite{j1} and \cite{hr}). 

We denote by $\omega$ the set of all non-negative integers (i.e., finite ordinal numbers of von Neumann). As usual, if $n\in\omega$, then $n+1=n\cup\{n\}$. Members of the set $\mathbb{N}=\omega\setminus\{0\}$ are called natural numbers. The power set of a set $X$ is denoted by $\mathcal{P}(X)$. The set of all finite subsets of $X$ is denoted by $[X]^{<\omega}$. A set $X$ is called \emph{countable} if $X$ is equipotent with a subset of $\omega$. An infinite countable set is called \emph{denumerable}.

The set of all real numbers is denoted by $\mathbb{R}$ and, if it is not stated otherwise, $\mathbb{R}$ is considered with its usual topology and with the metric induced by the standard absolute value.

\subsection{Basic definitions}
\label{s1.1}

 Let us recall basic definitions and some known facts concerning quasi-metrics and bitopological spaces. 
 
A \textit{bitopological space} is a triple $\langle X, \tau_1, \tau_2\rangle$ where $X$ is a set and $\tau_1, \tau_2$ are topologies in $X$. A \textit{quasi-pseudometric} on a set $X$ is a function $d: X\times X\to [0, +\infty)$ such that, for all $x,y,z\in X$, $d(x,y)\leq d(x,z)+d(z,y)$ and $d(x,x)=0$. A quasi-pseudometric $d$ on $X$ is called a \textit{quasi-metric} if, for all $x,y\in X$, the condition $d(x,y)=0$ implies $x=y$ (cf. \cite{fl}, \cite{Kel}, \cite{kunzi}, \cite{wils}). A quasi-(pseudo)metric $d$ on $X$ is a (\emph{pseudo})\emph{metric} if $d(x,y)=d(y,x)$ for all $x,y\in X$.  A (quasi)-(pseudo)metric $d$ on $X$ is called \emph{non-Archimedean} if $d(x,y)\leq\max\{d(x,z), d(z, y)\}$ for all $x,y,z\in X$.   A (\emph{quasi})-(\emph{pseudo})\emph{metric space} is an ordered pair $\langle X, d\rangle$ where $X$ is a set and $d$ is a (quasi)-(pseudo)metric on $X$.

Let $d$ be a quasi-(pseudo)metric on $X$. The \emph{conjugate} of $d$ is the quasi-(pseudo)metric $d^{-1}$ defined by: 
$$d^{-1}(x, y)=d(y, x) \text{ for } x, y\in X.$$
The (\emph{pseudo})\emph{metric} $d^{\star}$ \emph{associated with} $d$ is defined by: 
$$d^{\star}(x,y)=\max\{d(x,y), d(y,x)\} \text{ for } x,y\in X.$$  
The $d$-\textit{ball with centre $x\in X$ and radius} $r\in(0, +\infty)$ is the set 
$$B_{d}(x, r)=\{ y\in X: d(x, y)<r\}.$$
 The collection 
$$\tau(d)=\{ V\subseteq X: \forall_{x\in V}\exists_{n\in\omega} B_{d}(x, \frac{1}{2^n})\subseteq V\}$$
is the \textit{topology in $X$ induced by $d$}. The triple $\langle X,\tau(d), \tau(d^{-1})\rangle$ is \emph{ the bitopological space associated with $d$}.  For a set $A\subseteq X$, let $\delta_d(A)=0$ if $A=\emptyset$, and let $\delta_d(A)=\sup\{d(x,y): x,y\in A\}$ if $A\neq \emptyset$. Then $\delta_d(A)$ is the \emph{diameter} of $A$ in $\langle X, d\rangle$.

Clearly a quasi-(pseudo)metric $d$ on $X$ is a (pseudo)metric if and only if $d=d^{-1}=d^{\star}$.

\begin{remark}
\label{r01.2}
Some authors use the name ``quasi-metric'' in a distinct sense than above (see, for instance, \cite{rb}). We shall not be concerned with the notion of a quasi-metric given in \cite{rb}.
\end{remark}

For a topological space $\mathbf{X}=\langle X, \tau\rangle$ and a subset $A$ of $X$, the closure of $A$ in $\mathbf{X}$ is denoted by  $\text{cl}_{\tau}(A)$ or $\text{cl}_{\mathbf{X}}(A)$; moreover, $\tau\mid_A=\{A\in U: U\in\tau\}$ and $\mathbf{A}=\langle A, \tau\mid_A\rangle$ is the subspace of $\mathbf{X}$ with the underlying set $A$. In the sequel, boldface letters will denote  (quasi)-(pseudo)metric or topological spaces (called spaces in abbreviation) and lightface letters will denote their underlying sets.

\begin{definition}
\label{d1.1}
(Cf. pp. 74--75 of \cite{Kel}.) If $\tau, \tau_1, \tau_2$ be topologies in a set $X$, then:
\begin{enumerate}
\item[(i)] the bitopological space $\langle X, \tau_1, \tau_2\rangle$ is called (\emph{quasi})-(\emph{pseudo})\emph{metrizable} if there exists a (quasi)-(pseudo)metric $d$ on $X$ such that $\tau_1=\tau(d)$ and $\tau_2=\tau(d^{-1})$; 
\item[(ii)] the topological space $\langle X, \tau\rangle$ is called (\emph{quasi})-(\emph{pseudo})\emph{metrizable} if there exists a (quasi)-(pseudo)metric $d$ on $X$ such that $\tau(d)=\tau$.
\end{enumerate}
\end{definition}

One can find a considerable number of quasi-metrization theorems (unfortunately, proved mainly in $\mathbf{ZFC}$) in \cite{An} and in other sources (see, for instance, \cite{fl}).

\begin{definition}
\label{d01.2} A collection $\mathcal{U}$ of subsets of a space $\mathbf{X}$ is called:
\begin{enumerate} 
\item[(i)] \emph{locally finite} if every point of $X$ has a neighborhood meeting only finitely many members of $\mathcal{U}$;
\item[(ii)] \emph{point finite} if every point of $X$ belongs to at most finitely many members of $\mathcal{U}$;
\item[(iii)] \emph{interior-preserving} if, for every $x\in X$, the set $\bigcap\{U\in\mathcal{U}: x\in U\}$ is open in $\mathbf{X}$.
\item[(iv)] $\sigma$-\emph{locally finite} (resp., $\sigma$-\emph{point-finite}, $\sigma$-\emph{interior-preserving}) if $\mathcal{U}$ is a countable union of locally finite (resp., point-finite, interior-preserving) subfamilies. 
\end{enumerate}
\end{definition}

\begin{definition}
\label{d1.3} Let $\mathbf{X}$ be a space. Then:
\begin{enumerate}
 \item[(i)] $\mathbf{X}$ is \emph{first-countable} if every point of $X$ has a countable base of neighborhoods;
 \item[(ii)] $\mathbf{X}$ is \emph{second-countable} if $\mathbf{X}$ has a countable base;
 \item[(iii)] (cf., e.g, \cite{KV}, p.17) $\mathbf{X}$ is of \emph{countable pseudocharacter} if every singleton of $X$ is of type $G_{\delta}$ in $\mathbf{X}$;
 \item[(iv)] (cf., e.g., \cite{KV}, p. 16)  if $x\in X$, then a family $\mathcal{U}$ of open sets of $\mathbf{X}$ is called a \emph{pseudobase} at $x$ if $\{x\}=\bigcap\mathcal{U}$. 
 \end{enumerate}
\end{definition}

Given a collection  $\{X_j: j\in J\}$ of sets, for every $i\in J$, we denote by $\pi_i$ the projection $\pi_i:\prod\limits_{j\in J}X_j\to X_i$ defined by $\pi_i(x)=x(i)$ for each $x\in\prod\limits_{j\in J}X_j$. If $\tau_j$ is a topology in $X_j$, then $\mathbf{X}=\prod\limits_{j\in J}\mathbf{X}_j$ denotes the Tychonoff product of the topological spaces $\mathbf{X}_j=\langle X_j, \tau_j\rangle$ with $j\in J$. If  $\{\langle X_j, \tau_{1,j}, \tau_{2,j}\rangle: j\in J\}$ is a family of bitopological spaces, then $\prod\limits_{j\in J}\langle X_j, \tau_{1,j}, \tau_{2,j}\rangle$ is the bitopological space $\langle \prod\limits_{j\in J}X_j, \tau_1, \tau_2\rangle$ where $\tau_i$ is the topology of $\prod\limits_{j\in J}\langle X_{j}, \tau_{i,j}\rangle$ for $i\in\{1,2\}$.

We recall that if $\prod\limits_{j\in J}X_j\neq\emptyset$, then it is said that the family $\{X_j: j\in J\}$ has a choice function, and every element of $\prod\limits_{j\in J}X_j$ is called a \emph{choice function} of the family $\{X_j: j\in J\}$. A \emph{multiple choice function} of $\{X_j: j\in J\}$ is every function $f\in\prod\limits_{j\in J}\mathcal{P}(X_j)$ such that, for every $j\in J$, $f(j)$ is a non-empty finite subset of $X_j$. A set $f$ is called \emph{partial} (\emph{multiple}) \emph{choice function} of $\{X_j: j\in J\}$ if there exists an infinite subset $I$ of $J$ such that $f$ is a (multiple) choice function of $\{X_j: j\in I\}$. Given a non-indexed family $\mathcal{A}$, we treat $\mathcal{A}$ as an indexed family $\mathcal{A}=\{x: x\in\mathcal{A}\}$ to speak about a  (partial) choice function and a (partial) multiple choice function of $\mathcal{A}$.

We recall that if $\{\mathbf{X}_{n}=\langle X_{n},d_{n}\rangle:n\in 
\mathbb{N}\}$ is a family of (quasi)-(pseudo)metric spaces, then, for $X=\prod\limits_{n\in \mathbb{N}}X_{n}$, the function $d:X\times
X\rightarrow \mathbb{R}$ given by: 
\begin{equation}
d(x,y)=\sum\limits_{n\in \mathbb{N}}\frac{ \min\{d_n(x(n),y(n)), 1\}}{2^{n}}
\label{0}
\end{equation}%
for all $x,y\in X$, is
a (quasi)-(pseudo)metric on $X$ and the topology $\tau(d)$ in $X$ coincides with
the product topology of the family of spaces $\{\langle X_{n}, \tau(d_n)\rangle:n\in \mathbb{N%
}\}$ (cf., e.g., \cite{En} and  \cite{w}).

Let  $\{X_j: j\in J\}$ be a disjoint family of sets, that is, $X_i\cap X_j=\emptyset$ for each pair $i,j$ of distinct elements of $J$. If $\tau_j$ is a topology in $X_j$ for every $j\in J$, then $\bigoplus\limits_{j\in J}\mathbf{X}_j$ denotes the direct sum of the spaces $\mathbf{X}_j=\langle X_j, \tau_j\rangle$ with $j\in J$ (cf., e.g., \cite{En}). If  $\{\langle X_j, \tau_{1,j}, \tau_{2,j}\rangle: j\in J\}$ is a family of bitopological spaces, then $\bigoplus\limits_{j\in J}\langle X_j, \tau_{1,j}, \tau_{2,j}\rangle$ is the bitopological space $\langle \bigcup\limits_{j\in J}X_j, \tau_1, \tau_2\rangle$ where $\tau_i$ is the topology of $\bigoplus\limits_{j\in J}\langle X_{j}, \tau_{i,j}\rangle$ for $i\in\{1,2\}$.  Given a family $\{d_j: j\in J\}$ such that, for every $j\in J$, $d_j$ is a (quasi)-(pseudo)metric on $X_j$, one can define a (quasi)-(pseudo)metric $d$ on $X=\bigcup\limits_{j\in J}X_j$ as follows:
$$
(\ast) d(x,y)=\left\{ 
\begin{array}{c}
1 \text{ if there exist $i,j\in J$ such that } i\neq j, x\in X_i \text{ and } y\in X_j, \\ 
\min\{d_j(x,y), 1\}\text{ if there exists } j\in J \text{ such that }  x,y\in X_j.
\end{array}
\right.
$$
Then $\tau(d)$, where $d$ is defined by ($\ast$), coincides with the topology of the direct sum $\bigoplus\limits_{j\in j}\langle X_j, \tau(d_j)\rangle$, and the metric space $\langle \bigcup\limits_{j\in J}X_j, d\rangle$ is called the direct sum of the family $\{\langle X_j, d_j\rangle: j\in J\}$. In abbreviation, direct sums are called \emph{sums}.

\begin{definition}
\label{d1.5}
(Cf. \cite{hdkr}, \cite{hdhkr}, \cite{HowTach} and Form 419 in \cite{hr1}.) A \emph{cuf set} (resp., \emph{cuf family of sets}) is a set (resp., family) which is expressible as a countable union of finite sets (resp., finite subfamilies). 
\end{definition}

\begin{definition} 
\label{d1.6}
A \emph{cuf space} is a topological space $\langle X, \tau\rangle$ such that $X$ is a cuf set.
\end{definition}

\begin{definition}
\label{d1.7} 
 A \emph{cofinite space} is a topological space $\langle X, \tau\rangle$ such that 
 $$\tau=\{\emptyset\}\cup\{U\subseteq X: X\setminus U\in [X]^{<\omega}\},$$
that is, $\tau$  is the \emph{cofinite topology} in $X$.
A \emph{cofinite cuf space} is a cuf space equipped with its cofinite topology.
\end{definition}

\begin{definition}
\label{d1.8} 
Let $J$ be a non-empty set and let $\mathcal{X}=\{\mathbf{X}_j: j\in J\}$ be a family of topological spaces. 
\begin{enumerate}
\item[(i)] If $J$ is a cuf set, then the product $ \prod\limits_{j\in J}\mathbf{X}_j$ is called a \emph{cuf product} of the family $\mathcal{X}$.
\item[(ii)] If $J$ is a cuf set and the collection $\mathcal{X}$ is disjoint, then the direct sum $\bigoplus\limits_{j\in J}\mathbf{X}_j$ is called a \emph{cuf sum} of the family $\mathcal{X}$.
\end{enumerate}
\end{definition}

Analogously, we can introduce cuf products and cuf sums of bitopological spaces.

\begin{definition} 
\label{d1.9}
 A \emph{Hausdorff compactification} of a space $\mathbf{X}$ is an ordered pair $\langle \gamma\mathbf{X}, \gamma\rangle$ such that $\gamma\mathbf{X}$ is compact Hausdorff space and $\gamma$ is a homeomorphism of $\mathbf{X}$ onto a dense subspace of $\gamma\mathbf{X}$.
\end{definition}

\begin{remark}
\label{r1.10} 
Usually, a Hausdorff compactification $\langle \gamma\mathbf{X}, \gamma\rangle$  of $\mathbf{X}$ is denoted by $\gamma\mathbf{X}$, $\mathbf{X}$ is identified with the subspace $\gamma(\mathbf{X})$ of $\gamma\mathbf{X}$, and $\gamma$ is identified with the identity mapping on $X$. We denote the underlying set of $\gamma\mathbf{X}$ by $\gamma X$. Basic facts on Hausdorff compactifications in $\mathbf{ZFC}$ can be found in \cite{Ch}. Basic results on Hausdorff compactifications in $\mathbf{ZF}$ appeared in \cite{kerw}.
\end{remark}

\begin{definition}
\label{d01.11}
Suppose that $\gamma\mathbf{X}$ is a Hausdorff compactification of $\mathbf{X}$. 
\begin{enumerate} 
\item[(i)] The set $\gamma X\setminus X$ is called the \emph{remainder} of $\gamma\mathbf{X}$.
\item[(ii)] If $n\in\mathbb{N}$ and the remainder of $\gamma\mathbf{X}$ is an $n$-element set, then $\gamma\mathbf{X}$ is called an \emph{$n$-point Hausdorff compactification} of $\mathbf{X}$.
\item[(iii)] Every one-point Hausdorff compactification of $\mathbf{X}$ is called the \emph{Alexandroff compactification} of $\mathbf{X}$.
\end{enumerate}
\end{definition}

\begin{definition}
\label{d4.5}
(Cf. \cite{br}, \cite{lo} and \cite{kerta}.) 
\begin{enumerate}
\item[(i)] A space $\mathbf{X}$ is said to be \emph{Loeb} if $\mathbf{X}$ is empty or the family of all non-empty closed subsets of $\mathbf{X}$ has a choice function.
\item[(ii)] If $\mathbf{X}$ is a non-empty Loeb space, then every choice function of the family of all non-empty closed subsets of $\mathbf{X}$ is called a  \emph{Loeb function} of $\mathbf{X}$.
\end{enumerate}
\end{definition}

All other topological notions not defined here are standard and can be found in \cite{w}, \cite{En} and \cite{KV}.

\subsection{The list of weaker forms of $\mathbf{AC}$}
\label{s1.2}

In this subsection, for readers' convenience, we define and denote all weaker forms of $\mathbf{AC}$ used directly in this paper. For the known forms given in \cite{hr}, \cite{hr1} or \cite{hh}, we quote in their statements the form number under
which they are recorded in \cite{hr} (or in \cite{hr1} if they do not appear in \cite{hr}) and, if possible, we refer to their definitions in \cite{hh}. 

\begin{definition}
\label{d1.10}
\begin{enumerate}
\item $\mathbf{CAC}$ (Form 8 of \cite{hr}, Definition 2.5 in \cite{hh}): Every denumerable family of non-empty sets has a choice function.

\item  $\mathbf{CAC}(\mathbb{R})$ (Form 94 of \cite{hr}, Definition 2.9(1) in \cite{hh}): Every denumerable family of non-empty subsets of $\mathbb{R}$ has a choice function.

\item $\mathbf{CAC}_{fin}$ (Form 10 of \cite{hr}, Definition 2.9(3) in \cite{hh}): Every denumerable family of non-empty finite sets has a choice function.

\item $\mathbf{PC}(\aleph_0, 2, \infty)$ (Form [18 A] in \cite{hr}): Every denumerable family of two-element sets has a partial choice function. 

\item $\mathbf{CMC}$ (Form 126 of \cite{hr}, Definition 2.10 in \cite{hh}): Every denumerable family of non-empty sets has a multiple choice function.

\item $\mathbf{CMC}_{\omega }$ (Form 350 of \cite{hr}): Every denumerable family of denumerable sets has a multiple choice function.

\item $\mathbf{CUC}$ (Form 31 of \cite{hr}, Definition 3.2(1) in \cite{hh}): Every countable union of
countable sets is countable.

\item $\mathbf{CUC}_{fin}$ (Definition 3.2(3) in \cite{hh}): Every countable union of finite sets is countable.

\item $\mathbf{UT}(\aleph_0, cuf, cuf)$ (Form 419 in \cite{hr1}): Every countable union of cuf sets is a cuf set. (Cf. also \cite{hdhkr}.)

\item $\mathbf{UT}(\aleph_0, \aleph_0, cuf)$ (Form 420 in \cite{hr1}): Every countable union of countable sets is a cuf set. (Cf. also \cite{hdhkr}.)

\item $\mathbf{UT}(cuf, cuf, cuf)$: Every cuf union of cuf sets is a cuf set. (Cf. also \cite{hdhkr}.)

\item $\mathbf{vDCP}(\aleph_0)$ (Form 119 in \cite{hr}, p. 79 in \cite{hh}, \cite{vD}): Every denumerable family $\{\langle A_n, \leq_n\rangle: n\in\omega\}$ of linearly ordered sets, each of which is order-isomorphic to the set $\langle \mathbb{Z}, \leq\rangle$ of integers with the standard linear order $\leq$, has a choice function.

\item $\mathbf{BPI}$ (Form 14 in \cite{hr}, Definition 2.15(1) in \cite{hh}):  Every Boolean algebra has a prime ideal. 

\item $\mathbf{DC}$ (Form 43 in \cite{hr}, Definition 2.11(1) in \cite{hh}): For every non-empty set $X$ and every binary relation $\rho$ on $X$ if, for each $x\in X$ there exists $y\in X$ such that $x\rho y$, then  there exists a sequence $(x_n)_{n\in\mathbb{N}}$ of points of $X$ such that $x_n\rho x_{n+1}$ for each $n\in\mathbb{N}$.

\end{enumerate}
\end{definition}

\begin{remark}
\label{r1.11} (i) It is known that, in $\mathbf{ZF}$,  $\mathbf{CAC}_{fin}$ and $\mathbf{CUC}_{fin}$ are both equivalent in $\mathbf{ZF}$ (see Diagram 3.4 on page 23 in \cite{hh}) to the sentence: Every infinite well-ordered family of
non-empty finite sets has a partial choice function. (See Form [10 O] in \cite{hr} and Diagram 3.4 on p. 23 in \cite{hh}.) Moreover, 
 $\mathbf{CAC}_{fin}$ is equivalent (in $\mathbf{ZF}$) to Form [10 E] of \cite{hr}, that is, to the sentence: Every denumerable family of non-empty finite sets has a partial choice function.\smallskip
 
(ii) $\mathbf{CAC}$ is equivalent to the sentence: Every denumerable family of non-empty sets has a partial choice function. (See Form [8 A] in \cite{hr}.)\smallskip

(iii) It is also well known that $\mathbf{BPI}$ is equivalent to the statement that all products of compact Hausdorff spaces are compact (see Form [14 J] in \cite{hr} and Theorem 4.37 in \cite{hh}).\smallskip

(iv) It can be shown that $\mathbf{CMC}_{\omega}$ is equivalent to the following sentence: Every denumerable family of denumerable sets has a multiple choice function. $\mathbf{PC}(\aleph_0, 2, \infty)$ is equivalent to the sentence: Every denumerable disjoint family of 2-element sets has a choice function. (See Form 18 in \cite{hr}). 
\end{remark}

Let us pass to definitions of forms concerning metrizable spaces.

\begin{definition}
\label{d1.12}
\begin{enumerate}

\item $\mathbf{MP}$ (Form 383 in \cite{hr}): Every metrizable space is paracompact.

\item $\mathbf{M}(\sigma-p.f.)$ (Form 233 in \cite{hr}): Every metrizable space has a $\sigma$-point-finite base.

\item $\mathbf{M}(\sigma-l.f.)$ (Form [232 B] in \cite{hr}): Every metrizable space has a $\sigma$-locally finite base.

\item $\mathbf{CPM}(C, C)$: All countable products of compact metric spaces are compact.

\item $\mathbf{CPM}(C, S)$: All countable products of compact metric spaces are separable.

\item $\mathbf{CPM}_{le}$: Every countable product of metrizable spaces is
metrizable. (Cf. \cite{kt}.)

\item $\mathbf{CSM}_{le}$ (Form 4.18 in \cite{hr1}): Every countable sum of
metrizable spaces is metrizable. (Cf. \cite{kt}.)

\item $\mathbf{CPM}_{le}(C,M)$: Every countable product of compact metrizable spaces is metrizable.

\item $\mathbf{CSM}_{le}(C,M)$: Every countable sum of compact metrizable spaces is metrizable.
\end{enumerate}
\end{definition} 

Furthermore, in the following definition, we introduce new forms involving quasi-metrizability.

\begin{definition}
\label{d1.16}
\begin{enumerate}
\item $\mathbf{CPQM}$: All countable products of quasi-metrizable spaces are
quasi-metrizable.

\item $\mathbf{CSQM}$: Every countable sum of quasi-metrizable spaces is quasi- metrizable. 

\item $\mathbf{CPQMBi}$: All countable products of quasi-metrizable bitopological spaces are
quasi-metrizable.

\item $\mathbf{CSQMBi}$: All countable sums of
quasi-metrizable bitopological spaces are quasi-metrizable. 

\item $\mathbf{CPcofinQM}$: Every countable product of quasi-metrizable cofinite spaces is quasi-metrizable.

\item $\mathbf{CPcofin(\omega)QM}$: Every countable product of quasi-metrizable countable cofinite spaces is quasi-metrizable.
\end{enumerate}
\end{definition}

Other forms independent of $\mathbf{ZF}$ are defined in the forthcoming sections. 

\section{Introduction}
\label{s2}
\subsection{About the content of the paper}
\label{s2.1}

In an effort to demonstrate surprising disasters of topology in $\mathbf{ZF}$, it was established in many articles (see, e.g., \cite{vD}, \cite{gt}, \cite{HowTach}, \cite{kl}, \cite{kk}-\cite{kt}, \cite{t}-\cite{ew}) and books (see, e.g., \cite{hh}, \cite{hr}, \cite{j}) that a lot of everyday theorems of topology in $\mathbf{ZFC}$ may fail in some models of $\mathbf{ZF}$. In this paper, we concentrate mainly on several disasters concerning cuf products and cuf sums of (quasi)-metrizable spaces in $\mathbf{ZF}$. In \cite{ew}, it was noticed that $\mathbf{CPQM}$  implies van Douwen's principle $\mathbf{vDCP}(\aleph_0)$ and, in consequence $\mathbf{CPQM}$ is independent of $\mathbf{ZF}$. In the review of \cite{ew} in Zentralblatt Math. (see \cite{Tach1}), E. Tachtsis remarked that $\mathbf{CPQM}$ implies $\mathbf{CMC}_{\omega}$ and $\mathbf{CMC}_{\omega}$ is formally stronger than $\mathbf{vDCP}(\aleph_0)$. Being inspired by \cite{ew} and \cite{hdkr}, P. Howard and E. Tachtsis, have proved the following deep theorem in \cite{HowTach} recently:

\begin{theorem}
\label{t02.1}
(Cf. Theorems 4.1 and 4.2 in \cite{HowTach}). 
The following implications hold in $\mathbf{ZF}$ and none of them is reversible in $\mathbf{ZF}$:
 $$\mathbf{CPQM}\to \mathbf{UT}(\aleph_0, cuf, cuf)\to \mathbf{CMC}_{\omega}\to \mathbf{vDCP}(\aleph_0).$$
\end{theorem}

The following theorem was established in \cite{hdkr}:

\begin{theorem}
\label{t02.2} 
(Cf. \cite{hdkr}.) $\mathbf{UT}(\aleph_0, cuf, cuf)$ and $\mathbf{UT}(cuf, cuf, cuf)$ are equivalent in $\mathbf{ZF}$.
\end{theorem}

In Section \ref{s3}, among other results on cuf products and cuf sums of quasi-metrizable spaces, we show several new equivalences of $\mathbf{UT}(\aleph_0, cuf, cuf)$. In particular, we prove that $\mathbf{UT}(\aleph_0, cuf, cuf)$ and $\mathbf{CPcofinQM}$ are both equivalent to the sentence: Every cuf product of cofinite cuf spaces is first-countable.  We also prove that $\mathbf{UT}(\aleph_0, \aleph_0, cuf)$  and $\mathbf{CPcofin(\omega)QM}$ are both equivalent to the sentence: Every countable product of countable cofinite spaces is first-countable. Moreover, we show that $\mathbf{CPcofin(\omega)QM}$ implies $\mathbf{CMC}_{\omega}$. We also prove the sentence  ``For every family $\{\langle X_n, \tau_n\rangle: n\in\mathbb{N}\}$ of topological spaces, if there exists a sequence $(\mathcal{B}_n)_{n\in\mathbb{N}}$ such that each $\mathcal{B}_n$ is a $\sigma$-locally finite (resp., $\sigma$-point finite) base of $\langle X_n, \tau_n\rangle$, then $\prod\limits_{n\in\mathbb{N}}\langle X_n, \tau_n\rangle$ has a $\sigma$-locally finite (resp., $\sigma$-point finite) base'' follows from $\mathbf{CMC}$ and implies $\mathbf{UT}(\aleph_0, \aleph_0, cuf)$.  We notice that, in $\mathbf{ZF}$, a countable product of one-point Hausdorff compactifications of denumerable discrete spaces is (quasi)-metrizable if and only if it is first-countable. We show that the situation of countable products of two-point Hausdorff compactifications of denumerable discrete spaces is quite different. Namely, we prove that there is a model of $\mathbf{ZF}$ in which a countable product of two-point Hausdorff compactifications of denumerable discrete spaces can be first-countable without being quasi-metrizable. A list of open problems is included in Section \ref{s4}. In Subsection \ref{s4.2}, we give partial answers to two of the open problems by showing that it holds in $\mathbf{ZF}$ that every quasi-metrizable, compact Hausdorff Loeb space is metrizable and $\mathbf{CAC}$ implies that every quasi-metrizable compact Hausdorff space is metrizable.

\subsection{A list of several known theorems}
\label{s2.2}

We list below some known theorems for future references.

\begin{theorem}
\label{t02.5}
 (Cf. Theorem 3.3 in \cite{hdhkr}.) There exists a model of $\mathbf{ZF}$ in which $\mathbf{UT}(\aleph_0, \aleph_0, cuf)$ holds but $\mathbf{UT}(\aleph_0, cuf, cuf)$ fails.
\end{theorem}

\begin{theorem}
\label{t02.6}
(Cf. Theorem 2.1 in \cite{ew}.) $(\mathbf{ZF})$
If $J$ is a cuf set and $\{\langle X_j, d_j\rangle: j\in J\}$ is a family of (quasi)-metric spaces, then the product $\prod\limits_{j\in J}\langle X_j, \tau(d_j)\rangle$ is (quasi)-metrizable.
\end{theorem}

\begin{theorem}
\label{02.7} 
(Cf. Theorem 2.2 in \cite{ew}.) $(\mathbf{ZF})$ Let $\mathbf{X}$ be a (quasi)- metrizable space consisting of at
least two points. Then, for a set $J$, the following conditions are
equivalent:
\begin{enumerate}
\item[(i)] $\mathbf{X}^{J}$ is (quasi)-metrizable;
\item[(ii)] $\mathbf{X}^{J}$ is first-countable;
\item[(iii)] $J$ is a cuf set.
\end{enumerate}
\end{theorem}

\begin{theorem}
\label{t02.8} (Cf. Theorem 2.7 in \cite{ew}.) $(\mathbf{ZF})$ A cofinite space $\langle X, \tau\rangle$ is quasi-metrizable iff it is first-countable iff $X$ is a cuf set.
\end{theorem}

\begin{theorem}
\label{t02.12} (Cf. Corollary 4.8 in \cite{gt}.)  $(\mathbf{ZF})$ (Urysohn's Metrization Theorem.) If $%
\mathbf{X}$ is a second-countable $T_3$-space, then $\mathbf{X}$ is
metrizable.
\end{theorem}

\begin{theorem}
\label{t02.13}
 (Cf. \cite{naga}, \cite{sm}, \cite{bi}, \cite{cr}. ) 
 \begin{enumerate}
 \item[(i)]  $(\mathbf{ZFC})$ Every metrizable space has a $\sigma$-locally finite base.
 \item[(ii)] $(\mathbf{ZF})$ If a $T_1$-space $\mathbf{X}$ is regular and has a $\sigma$-locally finite base, then $\mathbf{X}$ is metrizable.
 \end{enumerate}
\end{theorem}

\begin{theorem}
\label{t02.14}
(Cf. Theorem 7.1 and Corollary 7.1 in \cite{fl} and p. 489 in \cite{KV}.) $(\mathbf{ZF})$ For every $T_1$-space $\mathbf{X}$, the following conditions are equivalent:
\begin{enumerate}
\item[(i)] $\mathbf{X}$ is non-Archimedeanly quasi-metrizable;
\item[(ii)] $\mathbf{X}$ has a $\sigma$-point-finite base;
\item[(iii)] $\mathbf{X}$ has a $\sigma$-interior-preserving base.
\end{enumerate}
\end{theorem}

\begin{remark}
\label{r02.14}
It is known from \cite{gt} that Urysohn's Metrization Theorem is provable in $%
\mathbf{ZF}$. That it holds in $\mathbf{ZFC}$ that a $T_1$-space is metrizable if and only if it is regular and has a $\sigma$-locally finite base was originally proved  by Nagata in \cite{naga}, Smirnov in \cite{sm} and Bing in \cite{bi}. It was shown in \cite{cr} that it is provable in $\mathbf{ZF}$ that every regular $T_1$-space which admits a $\sigma$-locally finite base is metrizable. On the other hand, it was established in \cite{hkrs} that $\mathbf{M}(\sigma-l.f.)$ (see Definition \ref{d1.12}(12)) is an equivalent to $\mathbf{M}(\sigma-p.f.)$ (see Definition \ref{d1.12}(11)) and implies $\mathbf{MP}$ (see Definition \ref{d1.12}(10)).
However, in \cite{gtw}, a model of $\mathbf{ZF}+\mathbf{DC}$ was shown in which $\mathbf{MP}$ fails. In \cite{sc}, a model of $\mathbf{ZF+BPI}$ was shown in which $\mathbf{MP}$ fails. All this taken together with Theorem \ref{t02.14} implies that, in each of the above-mentioned $\mathbf{ZF}$- models constructed in \cite{gtw} and \cite{sc},  there exists a metrizable space which fails to be non-Archimedeanly quasi-metrizable. On the other hand, it is a well-known fact of $\mathbf{ZFC}$ that every metrizable space is non-Archimedeanly quasi-metrizable (notice that this follows directly from Theorems \ref{t02.13}(i) and \ref{t02.14}); however, a metrizable space need not be non-Archimedeanly metrizable (see, e.g., p. 490 in \cite{KV}).
\end{remark}

\subsection{(Quasi)-metrics on countable unions}
\label{s2.3}

The following useful schema is a minor modification of an idea that has already appeared in some articles (cf., e.g. \cite{kerBull}) but we do not cite all of them here. \smallskip

Suppose that  $\mathcal{A}=\{A_n: n\in\mathbb{N}\}$ is a disjoint family of non-empty sets, $A=\bigcup\limits_{n\in\mathbb{N}}A_n$ and $\infty\notin A$. Let $X=A\cup\{\infty\}$.  Suppose we a given a sequence $(\rho_n)_{n\in\mathbb{N}}$ such that, for each $n\in\mathbb{N}$, $\rho_n$ is a (quasi)-metric on $A_n$. Let $d_n(x,y)=\min\{\rho_n(x,y), \frac{1}{n}\}$ for all $x,y\in A_n$. We define a function $d:X\times X\to\mathbb{R}$ as follows:
\[
(\ast\ast)\text{  }d(x,y)=\left\{ 
\begin{array}{c}
0 \text{ if } x=y, \\ 
\max\{\frac{1}{n}, \frac{1}{m}\}\text{ if } x\in A_n ,y\in A_m \text{ and } n\neq m,\\
d_n(x,y) \text{ if } x,y\in A_n,\\
\frac{1}{n}\text{ if } x\in A \text{ and } y=\infty\text{ or } x=\infty\text{ and } y\in A.
\end{array}
\right. 
\]

\begin{proposition}
\label{p02.16}
 The function $d$, defined by ($\ast\ast$), has the following properties:
\begin{enumerate}
\item[(i)] $d$ is a quasi-metric on $X$;
\item[(ii)] if, for every $n\in\mathbb{N}$, $\rho_n$ is a metric of $A_n$, then $d$ is a metric on $X$;
\item[(iii)] if, for every $n\in\mathbb{N}$, the space $\langle A_n, \tau(\rho_n)\rangle$ is compact, then so is the space $\langle X, \tau(d)\rangle$. 
\end{enumerate}
\end{proposition}

\section{The main results}
\label{s3}
\subsection{Preliminary remarks}
\label{s3.1}

To get an idea of the problems encountered while working with countable products (or sums) of ((quasi)-metrizable)
topological spaces in $\mathbf{ZF}$, let us consider a family $\{\mathbf{X}_n: n\in \omega\}$ of non-empty topological spaces and let $\mathbf{X}=\prod\limits_{n\in \omega}\mathbf{X}_{n}$. Clearly, in $\mathbf{ZFC}$, the space $\mathbf{X}$ is second-countable (resp,. separable) if and only if, for every $n\in\omega$, the space $\mathbf{X}_n$ is second-countable (resp., separable). However, the following propositions illustrate some problems arising in $\mathbf{ZF}$ but overlooked in $\mathbf{ZFC}$.

\begin{proposition}
\label{p02.18}
$(\mathbf{ZF})$ If all countable products of second-countable (resp., separable) metric spaces are second-countable (resp., separable), then $\mathbf{CUC}$ holds. 
\end{proposition}
\begin{proof}
Fix a disjoint family $\{A_n: n\in\omega\}$ of non-empty countable sets and a sequence $(\infty)_{n\in\omega}$ of pairwise distinct elements such that, for each $n\in\omega$, $\infty_n\notin\bigcup\limits_{m\in\omega}A_m$. For every $n\in\omega$, let $X_n=A_n\cup\{\infty_n\}$ and let $d_{n}$\ be the discrete metric on $X_{n}$. Obviously, for each $n\in\omega$, the discrete space $\mathbf{X}=\langle X_n,\mathcal{P}(X_n)\rangle$ is both second-countable and separable; moreover, $\mathcal{P}(X_n)=\tau(d_n)$. The product $\mathbf{X}=\prod\limits_{n\in\mathbb{N}}\mathbf{X}_n$ is metrizable by Theorem \ref{t02.6}. If $\mathbf{X}$ is separable, it is second-countable.  Suppose that $\mathbf{X}$ is second-countable. Then we can fix a countable base $\mathcal{B}=\{B_{n}:n\in \omega\}$ of $\mathbf{X}$. We notice that, for
every $n\in \omega$ and every $x\in A_{n}$, the set $\pi _{n}^{-1}(\{x\})$ is a non-empty open set
of $\mathbf{X}$, hence there exists $B\in \mathcal{B}$ with $B\subseteq \pi
_{n}^{-1}(\{x\})$. Therefore, given $n\in\omega$,  we can define a function $f_n: A_n\to\omega$ as follows:

$$f_n(x)=\min \{i\in \mathbb{N}:B_{i}\subseteq \pi _{n}^{-1}(\{x\})\} \text{ for every } x\in X_n.$$

\noindent Clearly, for every $n\in\omega$,  the function $f_{n}$ is an injection. Let $f:\bigcup\limits_{n\in\omega}A_n\to\omega\times\omega$ be defined by: $f(x)=\langle f_n(x), n\rangle$ if $x\in A_n$. Since $f$ is an injection into $\omega\times\omega$, the set $\bigcup\limits_{n\in\omega}A_n$ is countable. 
\end{proof}

Given a base $\mathcal{B}$ of a metrizable space $\mathbf{X}$, one may ask whether it is provable in $\mathbf{ZF}$ that, for every point $x$ of $\mathbf{X}$, there exists a countable pseudobase $\mathcal{U}$ at $x$ such that $\mathcal{U}\subseteq\mathcal{B}$. That the answer to this question is negative is shown by the following proposition whose proof is similar to that of Theorem 3.5 in \cite{gut} but we include its detailed proof for readers' convenience.

\begin{proposition}
\label{t03.17}
$(\mathbf{ZF})$
\begin{enumerate}
\item[(a)] $\mathbf{CAC}$ is equivalent to the following sentence:\newline 
$(\mathbf{PsB}(M))$: For every metrizable space $\mathbf{X}=\langle X, \tau\rangle$, for every base 
$\mathcal{B}$ of $\mathbf{X}$ and every $x\in X$, there is a countable subfamily $\mathcal{U}$
of $\mathcal{B}$ such that $\{x\}=\bigcap \mathcal{U}$.
\item[(b)]  $\mathbf{CAC}_{fin}$ follows from the following sentence:\newline
$(\mathbf{PsB}(CM))$: For every compact metrizable space $\mathbf{X}=\langle X, \tau\rangle$, for every base $\mathcal{B}$ of $\mathbf{X}$ and every $x\in X$, there is a countable subfamily $%
\mathcal{U}$ of $\mathcal{B}$ such that $\{x\}=\bigcap \mathcal{U}$.
\end{enumerate}
\end{proposition}

\begin{proof}
($a$) It is obvious that $\mathbf{CAC}$ implies $\mathbf{PsB}(M)$. For the converse implication, let us fix a disjoint family $\mathcal{A}=\{A_{n}:n\in \mathbb{N}\}$ non-empty sets. Take an element $\infty\notin\bigcup\limits_{n\in\mathbb{N}}A_n$ and put $X=\{\infty\}\cup\bigcup\limits_{n\in\mathbb{N}}A_n$. For every $n\in\mathbb{N}$ and each pair $x,y$ of elements of $A_n$, let $d_n(x,x)=0$ and $d_n(x,y)=\frac{1}{n}$ if $x\neq y$. Let $d$ be the metric defined by $(\ast\ast)$ in Subsection \ref{s2.2}. For each $m\in\mathbb{N}$ and $x\in A_m$, let 
$$V(x,m)=\{x, \infty\}\cup\bigcup\limits_{n=m+1}^{+\infty}A_n.$$
We notice that the family 
$$\mathcal{B}=\{\{a\}: a\in\bigcup\mathcal{A}\}\cup\{V(x, m): m\in\mathbb{N}, x\in A_m\}$$
\noindent is a base of the metrizable space $\mathbf{X}=\langle X, \tau(d)\rangle$. Suppose that $\mathcal{U}$ is a countable pseudobase at $\infty$ in $\mathbf{X}$ such that $\mathcal{U}\subseteq\mathcal{B}$. Let $\mathcal{U}=\{U_i: i\in\mathbb{N}\}$ and $U_i\neq U_j$ if $i\neq j$ ($i,j\in\mathbb{N}$). Then, for every $i\in\mathbb{N}$, there exist a unique $m_i\in\mathbb{N}$ and a unique $x_i\in A_{m_i}$ such that $U_i=V(x_i, m_i)$. Clearly, the family $\{A_{m_i}: i\in\mathbb{N}\}$ has a choice function, so $\mathcal{A}$ has a partial choice function. This, together with Remark \ref{r1.11}(ii), implies that $\mathbf{CAC}$ follows from $\mathbf{PsB}(M)$. 

($b$) Let us assume that, for every $n\in\mathbb{N}$, the set $A_n$ from the proof of ($a$) is finite. Then the space $\langle X, \tau(d)\rangle$ from the proof to ($a$) is a compact metrizable space. Therefore, it follows from the proof to ($a$) that $\mathbf{PsB}(CM)$ implies that every denumerable family of non-empty finite sets has a partial choice function. Hence $\mathbf{PsB}(CM)$ implies $\mathbf{CAC}_{fin}$ by Remark \ref{r1.11}(i).
\end{proof}

\begin{definition}
\label{d02.19}
Given a topological property $P$, we define the following sentences $\mathbf{Proj}(P)$, $\mathbf{Proj}(P, \neq\emptyset)$,  $\mathbf{Proj}(\omega, P)$ and $\mathbf{Proj}(P, \neq\emptyset)$:
\begin{enumerate}
\item $\mathbf{Proj}(P)$: For every non-empty set $J$ and every family $\{\mathbf{X}_j: j\in J\}$ of non-empty topological spaces, if $\mathbf{X}=\prod\limits_{j\in J}\mathbf{X}_j$ has $P$, then, for every $j\in J$, $\mathbf{X}_j$ has $P$.

\item $\mathbf{Proj}(P, \neq\emptyset)$: For every non-empty set $J$ and every family $\{\mathbf{X}_j: j\in J\}$ of non-empty topological spaces, if $\mathbf{X}=\prod\limits_{j\in J}\mathbf{X}_j$ is non-empty and has $P$, then, for every $j\in J$, $\mathbf{X}_j$ has $P$.

\item $\mathbf{Proj}(\omega, P)$: For every family $\{\mathbf{X}_n: n\in\omega\}$ of non-empty topological spaces, if $\mathbf{X}=\prod\limits_{n\in\omega}\mathbf{X}_n$ has $P$, then, for every $n\in\omega$, $\mathbf{X}_n$ has $P$.

\item $\mathbf{Proj}(\omega, P, \neq\emptyset)$: For every family $\{\mathbf{X}_n: n\in\omega\}$ of non-empty topological spaces, if $\mathbf{X}=\prod\limits_{n\in\omega}\mathbf{X}_n$ is non-empty has $P$, then, for every $n\in\omega$, $\mathbf{X}_n$ has $P$.
\end{enumerate}
\end{definition}

\begin{remark}
\label{r03.4}
Clearly, $\mathbf{Proj}(P)$ implies $\mathbf{Proj}(\omega, P)$ in $\mathbf{ZF}$. Many topological properties $P$ are such that $\mathbf{Proj}(P)$ holds in $\mathbf{ZFC}$ but $\mathbf{Proj}(P)$ may fail in a model of $\mathbf{ZF}$. For instance, in view of Propositions 2.13 and 2.14 of \cite{ew}, if $P$ is (quasi)-metrizability, then $\mathbf{Proj}(P)$ is equivalent to $\mathbf{AC}$, while $\mathbf{Proj}(\omega, P)$ is equivalent to $\mathbf{CAC}$. 
\end{remark}

The following proposition shows that, for instance,  compactness, separability, second-countability, first-countability are among topological properties $P$ such that $\mathbf{Proj}(\omega, P)$ fails in some models of $\mathbf{ZF}$, while, in some other models of $\mathbf{ZF}$, simultaneously $\mathbf{Proj}(\omega, P)$ holds and $\mathbf{Proj}(P)$ fails. 

\begin{proposition}
\label{p02.20}
$(\mathbf{ZF})$ Let $P$ be a topological property such that the empty space has $P$ and there exists a topological space which does not have $P$. Then:
\begin{enumerate}
\item[(i)]  $\mathbf{Proj}(P)$ implies $\mathbf{AC}$. 
\item[(ii)]  $\mathbf{Proj}(\omega, P)$ implies $\mathbf{CAC}$;
\item[(iii)] $\mathbf{Proj}(P, \neq\emptyset)$ implies that $\mathbf{Proj}(P)$ and $\mathbf{AC}$ are equivalent;
\item[(iv)]  $\mathbf{Proj}(\omega, P, \neq\emptyset)$ implies that $\mathbf{Proj}(\omega, P)$ and $\mathbf{CAC}$ are equivalent.
\end{enumerate}
\end{proposition}
\begin{proof} (i)-(ii) Suppose that $J$ is an infinite set and $\{A_j: j\in J\}$ is family of non-empty sets. For every $j\in J$, let $\mathbf{X}_j=\langle A_j, \mathcal{P}(A_j)\rangle$. Let $i$ be an element which is not in $J$ and let $\mathbf{X}_i$ be a topological space which does not have $P$. To complete the proof to (i) and (ii), it suffices to notice that if $\{A_j: j\in J\}$ does not have a choice function, then  $\mathbf{X}=\prod\limits_{j\in J\cup\{i\}}\mathbf{X}_j$ is the empty space, so $\mathbf{X}$ has $P$, while $\mathbf{X}_i$ does not. 
Clearly, (iii) follows from (i) and (iv) follows from (ii).
\end{proof}

\subsection{General basic theorems}
\label{s3.2}

Let us sketch a simple proof to the following basic theorem with new equivalences and implications (ii)-(vi):

\begin{theorem}
\label{t02.3}
$(\mathbf{ZF})$ 
\begin{enumerate}
\item[(i)] (Cf.  \cite{kt} and Exercises E1-E2 to Section 4.7 in \cite{hh}.)  $\mathbf{CPM}_{le}$ and $\mathbf{CSM}_{le}$ are both equivalent to the following sentence: For every family $\{\langle X_{n}, \tau_{n}\rangle: n\in \mathbb{N}\}$ of metrizable spaces, there exists a
family of metrics $\{d_{n}: n\in \mathbb{N}\}$ such that, for every $n\in \mathbb{N%
},\tau(d_{n})=\tau_{n}$.

\item[(ii)] $\mathbf{CPM}_{le}\mathbf{(}C,M)$ and $\mathbf{CSM}_{le}\mathbf{(}C,M)$ are both equivalent to the following sentence: For every family $\{\langle X_{n}, \tau_{n}\rangle: n\in \mathbb{N}\}$ of compact metrizable spaces, there exists a
family of metrics $\{d_{n}: n\in \mathbb{N}\}$ such that, for every $n\in \mathbb{N%
},\tau(d_{n})=\tau_{n}$.
 
\item[(iii)] $\mathbf{CPQM}$ and $\mathbf{CSQM}$ are both equivalent to the following sentence: For every family $\{\langle X_{n}, \tau_{n}\rangle: n\in \mathbb{N}\}$ of quasi-metrizable spaces, there exists a
family of quasi-metrics $\{d_{n}: n\in \mathbb{N}\}$ such that, for every $n\in \mathbb{N%
},\tau(d_{n})=\tau_{n}$.

\item[(iv)] $\mathbf{CPQMBi}$ and $\mathbf{CSQMBi}$ are both equivalent to the following sentence: For every family $\{\langle X_{n}, \tau_{1,n}, \tau_{2,n}\rangle: n\in \mathbb{N}\}$ of quasi-metrizable bitopological spaces, there exists a
family of quasi-metrics $\{d_{n}: n\in \mathbb{N}\}$ such that, for every $n\in \mathbb{N}$ and $i\in\{1,2\}$, $\tau(d_{i,n})=\tau_{i,n}$.
\item[(v)] $\mathbf{CPQMBi}$ implies $\mathbf{CPM}_{le}$. 
\item[(vi)] $\mathbf{CPQM}$ implies $\mathbf{CPcofinQM}$ implies $\mathbf{CPcofin(\omega)QM}$. 
\end{enumerate}
\end{theorem}
\begin{proof} Similarly to the proof to (i) in \cite{kt}, we can prove (ii),(iii) and (iv). To show (v), it suffices to notice that a topological space $\langle X, \tau\rangle$ is metrizable if and only if the bitoplogical space $\langle X, \tau, \tau\rangle$ is quasi-metrizable. The implications of (vi) are obvious. 
\end{proof}

\begin{remark}
\label{r02.4}
 That none of the implications of Theorem \ref{t02.3}(vi) is reversible is shown in Theorems \ref{t03.18} and \ref{t3.20} in Subsection \ref{s3.3}. Open problems concerning some implications relevant to Theorem \ref{t02.3} are posed in Section \ref{s4}.
\end{remark}

The following proposition allows us to reduce a considerable number of problems and to extend some theorems we are interested in. 

\begin{proposition}
\label{p03.1} $(\mathbf{ZF})$ Let $P$ be a topological property.
\begin{enumerate}
\item[(i)] If $P$ is countably multiplicative (i.e., if every countable product of spaces possessing $P$ has  $P$), then  $P$ is cuf multiplicative, i.e., every cuf product of spaces possessing $P$ has $P$.
\item[(ii)] If $P$ is countably additive (i.e., if every countable direct sum of spaces possessing $P$ has  $P$), then  $P$ is cuf additive, i.e., every cuf sum of spaces possessing $P$ has $P$.
\end{enumerate}
\end{proposition}
\begin{proof}
Suppose that $\{J_n: n\in\omega\}$ is a disjoint family of non-empty finite sets. Let $J=\bigcup\limits_{n\in\omega}J_n$. For each $j\in J$, let $\langle X_j, \tau_j\rangle$ be a topological space. To complete the proof, we notice that the cuf product $\prod\limits_{j\in J}\langle X_j, \tau_j\rangle$ is homeomorphic with the countable product $\prod\limits_{n\in\omega}(\prod\limits_{j\in J_n}\langle X_j, \tau_j\rangle)$, while the direct cuf sum $\bigoplus\limits_{j\in J}\langle X_j, \tau_j\rangle$ is homeomorphic with the countable direct sum $\bigoplus\limits_{n\in\omega}(\bigoplus\limits_{j\in J_n}\langle X_j, \tau_j\rangle)$.
\end{proof}
 
From Theorem \ref{t02.3} and  Proposition \ref{p03.1} (or its proof), we immediately deduce the following:

\begin{theorem}
\label{c03.2}
$(\mathbf{ZF})$
\begin{enumerate} 
\item[(i)]  $\mathbf{CPM}_{le}$ iff all cuf products of metrizable spaces are metrizable iff all cuf sums of metrizable spaces are metrizable.
\item[(ii)] $\mathbf{CPQM}$ iff all cuf products of quasi-metrizable spaces are quasi-metri\-zable iff all cuf sums of quasi-metrizable spaces are quasi-metrizable.
\item[(iii)] $\mathbf{CPQMBi}$ iff all cuf products of quasi-metrizable bitopological spaces are quasi-metrizable iff all cuf sums of quasi-metrizable bitopological spaces are quasi-metrizable.
\item[(iv)] $\mathbf{CPM}(C,C)$ iff all cuf products of compact metric spaces are compact.
\item[(v)]$\mathbf{CPM}(C,S)$ iff all cuf products of compact metric spaces are separable.
\end{enumerate}
\end{theorem}

Using Theorem \ref{t02.6} and arguing in much the same way, as in the proof to Theorem \ref{t02.3}, we can prove the following proposition:

\begin{proposition}
\label{p03.3}
let $J$ be a non-empty cuf set and let $\{\langle X_j, \tau_j\rangle: j\in J\}$ be a family of pairwise disjoint topological spaces such that $\prod\limits_{j\in J}X_j\neq\emptyset$. Then $\bigoplus\limits_{j\in J}\langle X_j, \tau_j\rangle$ is (quasi)-metrizable if and only if $\prod\limits_{j\in J}\langle X_j, \tau_j\rangle$ is (quasi)-metrizable.
\end{proposition}

In view of Proposition \ref{p03.3}, searching for conditions under which countable products and countable sums are (quasi)-metrizable, we can limit our attention to countable products.  

\subsection{Cuf products of cofinite cuf spaces and one-point Hausdorff compactifications of infinite discrete cuf spaces}
\label{s3.3}

\begin{definition}
\label{d03.4} 
Let $X$ be an infinite set, $\infty\notin X$ and $X(\infty)=X\cup\{\infty\}$. Then $\mathbf{X}(\infty)=\langle X(\infty), \tau\rangle$ where $$\tau=\mathcal{P}(X)\cup\{U\subseteq X(\infty): \infty\in U\text{ and } X\setminus U\in [X]^{<\omega}\}.$$  
\end{definition}

\begin{remark}
\label{r03.5} 
Clearly, for an infinite set $X$, the space $\mathbf{X}(\infty)$ is the one-point Hausdorff compactification of the discrete space $\langle X, \mathcal{P}(X)\rangle$. If $X=Y_n$, then we denote $\mathbf{X}(\infty)$ by $\mathbf{Y}_n(\infty)$. 
\end{remark}

In what follows, we use the notions of a pseudobase and a countable pseudocharacter, given in Definition \ref{d1.3}.

\begin{definition}
\label{d3.6} 
We say that a space $\mathbf{X}$ is of cuf pseudocharacter if, for every point of $\mathbf{X}$, there exists (in $\mathbf{X}$) a pseudobase $\mathcal{U}$ at $x$ such that $\mathcal{U}$ is a cuf set.
\end{definition}

\begin{remark}
\label{r3.7}
(i) Of course, every first-countable $T_1$-space is of countable pseudocharacter. Every space of countable pseudo\-character is of cuf pseudo\-character. Every (quasi)-metrizable space is first-countable.\smallskip

(ii) A known result of $\mathbf{ZF}+\mathbf{CAC}$ is that every compact Hausdorff space of countable pseudocharacter is first-countable. It occurs that a compact Hausdorff space of countable pseudocharacter may fail to be first-countable in a model of $\mathbf{ZF}$. Namely, in \cite{kt1} and \cite{tach3}, a model $\mathcal{M}$ of $\mathbf{ZF}$ was shown in which there exists a denumerable compact Hausdorff space $\mathbf{X}$ which does not have a denumerable family of pairwise disjoint non-empty open sets. Clearly, this space $\mathbf{X}$ is not first-countable but it is of countable pseudocharacter.
\end{remark}

\begin{proposition}
\label{p03.6} $(\mathbf{ZF})$ Let $X$ be an infinite set and let $\mathbf{X}=\langle X, \mathcal{P}(X)\rangle$. The following are equivalent:
\begin{enumerate}
\item[(i)] the one-point Hausdorff compactification of $\mathbf{X}$ is (non-Archimedeanly) metrizable;
\item[(ii)] the one-point Hausdorff compactification of $\mathbf{X}$ is of cuf pseudocharacter;
\item[(iii)] $X$ is a cuf set.
\end{enumerate}
\end{proposition}

\begin{proof} In view of Remark \ref{r3.7}(i), the implications $(i)\rightarrow (ii)$ is obvious. Let $\infty$ be an element such that $\infty\notin X$. To show that $(ii)$ implies $(iii)$, suppose that, in $\mathbf{X}(\infty)$, $\mathcal{U}$ is a pseudobase at $\infty$ such that $\mathcal{U}$ is a cuf set. Let $\mathcal{U}=\bigcup\limits_{n\in\mathbb{N}}\mathcal{U}_n$ where, for every $n\in\mathbb{N}$, the family $\mathcal{U}_n$ is finite.  Then, for every $n\in\mathbb{N}$, the set $F_n=\bigcup\limits_{U\in\mathcal{U}_n}(X\setminus U)$ is finite. Since $X=\bigcup\limits_{n\in\mathbb{N}}F_n$, we deduce that $X$ is a cuf set. Hence $(ii)$ implies $(iii)$.

Now, suppose that $(iii)$ holds. Let $\{A_n: n\in\omega\}$ be a collection of non-empty finite sets such that $X=\bigcup\limits_{n\in\mathbb{N}}A_n$ and $A_{n}\subseteq A_{n+1}$ for each $n\in\mathbb{N}$. For $x\in X$, let $n(x)=\min\{n\in\omega: x\in A_{n}\}$. For $x,y\in X\cup\{\infty\}$, we define
\[
d(x,y)= d(y,x)=\left\{ 
\begin{array}{c}
0 \text{ if } x=y, \\ 
\max\{\frac{1}{n(x)}, \frac{1}{n(y)}\}\text{ if } x,y\in X \text{ and } x\neq y,\\
\frac{1}{n(x)}\text{ if } x\in X\text{ and }y=\infty.
\end{array}
\right. 
\]
Then $d$ is a non-Archimedean metric on $X\cup\{\infty\}$ which induces the topology of $\mathbf{X}(\infty)$. Hence $(iv)$ implies $(i)$.
\end{proof}

\begin{lemma}
\label{l03.7} $(\mathbf{ZF})$  
Let $\{A_n: n\in\mathbb{N}\}$ be a collection of finite sets such $A_n\subseteq A_{n+1}$ for each $n\in\mathbb{N}$. Suppose that the set $X=\bigcup\limits_{n\in\mathbb{N}}A_n$ is non-empty. For $x\in X$, let $n(x)=\min\{n\in\mathbb{N}: x\in A_n\}$. For $x,y\in X$ with $x\neq y$, let $\rho(x,y)$ be defined as follows: 
\[
\rho(x,y)= \left\{ 
\begin{array}{c}
0 \text{ if } x=y,\\ 
\frac{1}{n(y)}\text{ if }  x\neq y.
\end{array}
\right.
\]
 Then $\rho$  is a non-Archimedean quasi-metric on $X$ such that $\rho$ induces the cofinite topology of $X$.
\end{lemma}

\begin{proof} The proof is similar to that of Theorem 2.7 in \cite{ew}. We include it for readers' convenience. We notice that if $x,y,z\in X$, then $\rho(x, y)\leq\max\{\rho(x, z), \rho(z, y)\}$, so $\rho$ is a non-Archimedean quasi-metric. Let $\tau$ be the cofinite topology on $X$.  Let $x\in X$ and $n\in\mathbb{N}$.  If $n\in\mathbb{N}$ and $y\in X\setminus B_{\rho}(x,\frac{1}{n+1})$, then $\rho(x,y)\ge\frac{1}{n+1}$, so $n+1\leq n(y)$ and this implies that $X\setminus B_{\rho}(x,\frac{1}{n+1})\subseteq A_n$. Hence $B_{\rho}(x, \frac{1}{n+1})\in \tau$. On the other hand, given $U\in\tau$, we can fix $m\in\mathbb{N}$ such that $X\setminus U\subseteq A_m$. Then, for $x\in U$, we have $B_{\rho}(x,\frac{1}{m})\subseteq U$. 
\end{proof}

\begin{theorem}
\label{t03.8}
$(\mathbf{ZF})$
 Let $J$ be a non-empty cuf set and let $\{\langle X_j, \tau_j\rangle: j\in J\}$ be a collection of cofinite spaces. Then the product $\mathbf{X}=\prod\limits_{j\in J}\langle X_j, \tau_j\rangle$ is quasi-metrizable iff it is of cuf pseudocharacter.
\end{theorem}
\begin{proof} By virtue of Remark \ref{r3.7}(i), it suffices to prove that if $\mathbf{X}$ is of cuf pseudocharacter, then $\mathbf{X}$ is quasi-metrizable. Let $X=\prod\limits_{j\in J}X_j$. If $X=\emptyset$, then $\mathbf{X}$ is quasi-metrizable.  Suppose that $X\neq\emptyset$. Then we can fix a point $z\in X$. Suppose that, in $\mathbf{X}$, $\mathcal{U}$ is a pseudobase at $z$ such that $\mathcal{U}$ is a cuf set. Let $\mathcal{U}=\bigcup\limits_{n\in\mathbb{N}}\mathcal{U}_n$ where, for every $n\in\mathbb{N}$, the family $\mathcal{U}_n$ is finite. Consider any $i\in J$ and  $U\in\mathcal{U}$. Let $E(i,U)$ be the set of all $t\in X_i$ such that there exists $y\in X\setminus U$ such that $y(i)=t$ and $y(j)=z(j)$ for every $j\in J\setminus\{i\}$. Then, the set $E(i, U)$ is closed in $\langle X_i, \tau_i\rangle$ and $E(i, U)\neq X_i$,  so $E(i,U)$ is a finite subset of $X_i$. Hence, for every $n\in\mathbb{N}$ and every $i\in J$, the set $A_{i,n}=\bigcup\limits_{U\in\mathcal{U}_n}E(i, U)$ is finite. In this way, we have defined a collection $\{A_{j,n}: j\in J, n\in\mathbb{N}\}$ of finite sets such that $X_j\setminus\{z(j)\}=\bigcup\limits_{n\in\mathbb{N}}A_{j,n}$ for each $j\in J$. Now, we can use Lemma \ref{l03.7} to define a collection $\{\rho_j: j\in J\}$ of quasi-metrics such that $\tau_j=\tau(\rho_j)$ for each $j\in J$. By Theorem \ref{t02.6}, $\mathbf{X}$ is quasi-metrizable.
\end{proof}

From Theorem \ref{t03.8} and Proposition \ref{p03.3}, we can immediately deduce the following corollary:

\begin{corollary}
\label{c03.9}
$(\mathbf{ZF})$  Let $J$ be a non-empty cuf set and let $\{\langle X_j, \tau_j\rangle: j\in J\}$ be a family of pairwise disjoint cofinite spaces such that $\prod\limits_{j\in J}X_j\neq\emptyset$. Then $\bigoplus\limits_{j\in J}\langle X_j, \tau_j\rangle$ is quasi-metrizable iff $\prod\limits_{j\in J}\langle X_j, \tau_j\rangle$ is of cuf pseudocharacter.
 \end{corollary}
 
 \begin{theorem}
\label{t03.10}
$(\mathbf{ZF})$ Let $J$ be a non-empty cuf set and let $\{\mathbf{Y}_j: j\in J\}$ be a family of one-point Hausdorff compactifications of infinite discrete spaces. Then the product $\mathbf{Y}=\prod\limits_{j\in J}\mathbf{Y}_j$ is metrizable iff it is quasi-metrizable iff it is of cuf pseudocharacter.
\end{theorem}
\begin{proof} 
It suffices to prove that if $\mathbf{Y}$ is of cuf pseudocharacter, then $\mathbf{Y}$ is metrizabble. For every $j\in J$, let $\infty_j$ be the unique accumulation point of $\mathbf{Y}_j$ and let $X_j=Y_j\setminus\{\infty_j\}$. Then $\mathbf{Y}_j=\mathbf{X}_j(\infty_j)$. Let $z\in\prod\limits_{j\in J}X(\infty_j)$ be defined by: $z(j)=\infty_j$ for every $j\in J$. Assuming that $\mathbf{Y}$ is of cuf pseudocharacter, we fix (in $\mathbf{Y}$) pseudobase $\mathcal{U}$ at $z$ such that $\mathcal{U}$ is a cuf set. Let $\mathcal{U}=\bigcup\limits_{n\in\mathbb{N}}\mathcal{U}_n$ where, for every $n\in\mathbb{N}$, $\mathcal{U}_n$ is a finite family.  Mimicking the proof to Theorem \ref{t03.8}, we define a collection $\{A_{j,n}: j\in J, n\in\mathbb{N}\}$ of finite sets such that, for every $j\in J$,  $X_j=\bigcup\limits_{n\in\mathbb{N}}A_{j,n}$. We may assume that, for every $j\in J$ and every $n\in\mathbb{N}$, $A_{j,n}\subseteq A_{j,n+1}$. In much the same way, as in the proof to Proposition \ref{p03.6}, we define a collection $\{\rho_j: j\in J\}$ such that, for every $j\in J$, $\rho_j$ is a metric on $X_j(\infty_j)$ which induces the topology of $\mathbf{X}_j(\infty_j)$. This, together with Theorem \ref{t02.6}, implies that $\mathbf{Y}$ is metrizable. 
\end{proof}

The following corollary is a consequence of Theorem \ref{t03.10} and Proposition \ref{p03.3}. 

\begin{corollary}
\label{c03.11}
$(\mathbf{ZF})$ Let $J$ be a non-empty cuf set and let $\{\mathbf{Y}_j: j\in J\}$ be a disjoint family of one-point Hausdorff compactifications of infinite discrete spaces. Then the following conditions are equivalent:
\begin{enumerate}
\item[(i)] $\bigoplus\limits_{j\in J}\mathbf{Y}_j$ is metrizable
\item[(ii)] $\bigoplus\limits_{j\in J}\mathbf{Y}_j$ is quasi-metrizable
\item[(iii)] $\prod\limits_{j\in J}\mathbf{Y}_j$ is of cuf pseudocharacter. 
\end{enumerate}
\end{corollary}
 
\begin{theorem}
\label{t03.12}
$(\mathbf{ZF})$ If every  countable product of compact  metrizable cuf spaces is of cuf pseudocharacter, then every  cuf product of cofinite cuf spaces is quasi-metrizable.
\end{theorem}

\begin{proof}
Let $J$ be a non-empty cuf set and let $\{\langle X_{j},\tau_{j}\rangle: j\in J\}$ be a family of non-empty cuf spaces such that, for every $j\in J$, $\tau_{j}$ is the cofinite topology in $%
X_{j}$. Let $\infty\notin\bigcup\limits_{j\in J}X_j$. We put $Y_j=X_j\cup\{\infty\}$. If $X_j$ is finite, let $\mathbf{Y}_j$ be the discrete space $\langle Y_j, \mathcal{P}(Y_j)\rangle$. If $Y_j$ is infinite, let $\mathbf{Y}_j=\mathbf{X}_j(\infty)$. It follows from Proposition \ref{p03.6} that, for every $j\in J$, the space $\mathbf{Y}_j$ is a compact metrizable cuf space. Assume that all countable products of compact metrizable cuf spaces are of cuf pseudocharacter. Then it follows from the proof to Proposition \ref{p03.1} that the product $\mathbf{Y}=\prod\limits_{j\in J}\mathbf{Y}_j$ is of cuf pseudocharacter. Let $z\in\prod\limits_{j\in J}Y_j$ be defined by: $z(j)=\infty$ for each $j\in J$. In $\mathbf{Y}$. We fix a pseudobase $\mathcal{U}$ at $z$ such that $\mathcal{U}$ is a cuf set. One can mimic the proofs to Theorems \ref{t03.8} and \ref{t03.10} to find a collection $\{A_{j,n}: j\in J, n\in\mathbb{N}\}$ of finite sets such that, for every $j\in J$,  $X_j=\bigcup\limits_{n\in\mathbb{N}}A_{j,n}$. Now, it follows from Lemma \ref{l03.7} that there exists a collection $\{\rho_j: j\in J\}$ of quasi-metrics such that $\tau(\rho_j)=\tau_j$ for every $j\in J$. In the light of Theorem \ref{t02.6}, the product $\prod\limits_{j\in J}\langle X_j, \tau_j\rangle$ is quasi-metrizable.
\end{proof}

\begin{theorem}
\label{t03.14} $(\mathbf{ZF})$ If every cuf product of cofinite cuf spaces is of cuf pseudocharacter, then every cuf product of one-point Hausdorff compactifications of infinite discrete cuf spaces is metrizable.
\end{theorem}

\begin{proof} Let $J$ be a non-empty cuf set and let $\{\mathbf{Y}_j: j\in J\}$ be a collection of one-point Hausdorff compactifications of infinite discrete cuf spaces. Similarly to the proof of Theorem \ref{t03.10}, for every $j\in J$, let $\infty_j$ be the unique accumulation point of $\mathbf{Y}_j$ and let $X_j=Y_j\setminus\{\infty_j\}$. Then $\mathbf{Y}_j=\mathbf{X}_j(\infty_j)$. For every $j\in J$, let $\tau_j$ be the cofinite topology in $Y_j$. Suppose that the space $\mathbf{X}=\prod\limits_{j\in J}\langle Y_j, \tau_j\rangle$ is of cuf pseudocharacter. In much the same way, as in the proofs to Theorems \ref{t03.10} and \ref{t03.12}, we can define a family $\{ A_{j,n}: j\in J, n\in\mathbb{N}\}$ of finite sets such that $X_j=\bigcup\limits_{n\in\mathbb{N}}A_{j,n}$ for every $j\in J$. Arguing similarly to the proof that (iv) implies (i) in Proposition \ref{p03.6}, we can define a collection $\{d_j: j\in J\}$ such that, for every $j\in J$, $d_j$ is a metric which induces the topology of $\mathbf{Y}_j$. It follows from Theorem \ref{t02.6} that $\prod\limits_{j\in J}\mathbf{Y}_j$ is metrizable.
\end{proof}

Using similar arguments, one can prove the following theorem:

\begin{theorem} 
\label{t03.15}
$(\mathbf{ZF})$ If every countable product of countable cofinite spaces is of cuf pseudocharacter, then every countable product of one-point Hausdorff compactifications of denumerable discrete spaces is metrizable.
\end{theorem}

\begin{theorem}
\label{t03.18}
$(\mathbf{ZF})$ 
\begin{enumerate}
\item[(i)] $\mathbf{UT}(cuf, cuf, cuf)$ and $\mathbf{CPcofinQM}$ are both  equivalent to the following sentence: Every cuf product of cofinite cuf spaces is first-countable.
\item[(ii)] $\mathbf{UT}(\aleph_0, \aleph_0, cuf)$ and $\mathbf{CPcofin(\omega)QM}$ are both equivalent to the following sentence: Every cuf product of countable cofinite spaces is first-countable.
\item[(iii)] $\mathbf{CPcofin(\omega)QM}$ does not imply $\mathbf{CPcofinQM}$.
\end{enumerate}
\end{theorem}
\begin{proof}
We fix a cuf set $J$ and a family $\{X_j: j\in J\}$ of cofinite cuf sets. For every $j\in J$, let $\tau_j$ be the cofinite topology in $X_j$, let $\mathbf{X}_j=\langle X_j, \tau_j\rangle$, $X=\prod\limits_{j\in J}X_j$ and $\mathbf{X}=\prod\limits_{j\in J}\mathbf{X}_j$ . We assume that $J=\bigcup\limits_{n\in\omega}J_n$ where $\{J_n: n\in\omega\}$ is a disjoint family of non-empty finite sets. For every $n\in\omega$, let $\mathbf{Y}_n=\prod\limits_{j\in J_n}\mathbf{X}_j$ and $\mathbf{Y}=\prod\limits_{n\in\omega}\mathbf{Y}_n$. Then $\mathbf{X}$ and $\mathbf{Y}$ are homeomorphic.  For every $n\in\omega$, let $E_n=\bigcup\limits_{j\in J_n}X_j$. Then, for every $n\in\omega$,  $E_n$ is a cuf set and, moreover, if $X_j$ is countable for every $j\in J_n$, then $E_n$ is countable. Suppose that $E=\bigcup\limits_{n\in\omega}E_n$ is a cuf set. Let $E=\bigcup\limits_{n\in\omega}A_n$ where, for every $n\in\omega$, $A_n$ is a finite set. For fixed  $j\in J$ and $x\in X_j$, we define a countable collection $\mathcal{B}_j(x)$ as follows:
$$\mathcal{B}_j(x)=\{\{ x\}\cup (X_j\setminus\bigcup\limits_{i\in k+1}A_i): k\in\omega\}.$$
It is obvious that $\mathcal{B}_j(x)\subseteq\tau_j$. On the other hand, if $x\in U\in\tau_j$, then $X_j\setminus U$ is a finite set, so there exists $k_U\in\omega$ such that $X_j\setminus U\subseteq\bigcup\limits_{i\in k_U+1}A_i$. Then $V=\{x\}\cup(X_j\setminus\bigcup\limits_{i\in k_U+1}A_i)\in\mathcal{B}_j(x)$ and $V\subseteq U$. This proves that $\mathcal{B}_j(x)$ is a countable base of neighborhoods of $x$ in $\mathbf{X}_j$. Using the collections $\mathcal{B}_j(x)$, for every point $y$ of $\mathbf{Y}$, we can define a collection $\{U_{n,i}(y(n)): n,i\in\omega\}$ such that, for every $n\in\omega$, the collection $\{U_{n,i}(y(n)): i\in\omega\}$ is a base of neighborhoods of $y(n)$ in  $\mathbf{Y}_n$. This proves that $\mathbf{Y}$ is first-countable if $E$ is a cuf set. Hence $\mathbf{X}$ is first-countable if $E$ is a cuf set. Furthermore, it follows from Theorem \ref{t03.8} that if $E$ is a cuf set, then $\mathbf{X}$ is quasi-metrizable. Hence, $\mathbf{UT}(cuf,cuf, cuf)$ implies ``Every cuf product of cofinite cuf spaces is first-countable'' implies $\mathbf{CPcofinQM}$; moreover,  $\mathbf{UT}(\aleph_0, \aleph_0, cuf)$ implies ``Every cuf product of countable cofinite spaces is first-countable'' implies $\mathbf{CPcofin(\omega)QP}$.

Now, suppose that $J=\omega$ and assume $\mathbf{CPcofinQM}$. We prove that $\bigcup\limits_{n\in\omega}X_n$ is a cuf set. We notice that, given a sequence $(\infty_n)_{n\in\omega}$ of elements, if $\bigcup\limits_{n\in\omega}(X_n\cup\{\infty_n\})$ is a cuf set, then $\bigcup\limits_{n\in\omega}X_n$ is a cuf set. Hence,  without loss of generality, we may assume that we have fixed a sequence $(\infty_n)_{n\in\omega}$ such that $\infty_n\in X_n$ for each $n\in\omega$. Let $z\in X$ be defined by: $z(n)=\infty_n$ for each $n\in\omega$. By $\mathbf{CPcofinQM}$, the space $\mathbf{X}$ is quasi-metrizable, so it is first-countable. Let $\mathcal{V}$ be a countable base of neighborhoods of $z$ in $\mathbf{X}$ and, for every $i\in\omega$, let $\pi_i:X\to X_i$ be the projection. We notice that the collection $\mathcal{E}=\{X_n\setminus\pi_n(V): n\in\omega, V\in\mathcal{V}\}\cup\{\{\infty_n\}: n\in\omega\}$ is countable and consists of finite sets. To check that $\mathcal{E}$ is a cover of $\bigcup\limits_{n\in\omega}X_n$, we fix $i\in\omega$ and a point $t\in X_i\setminus\{\infty_i\}$. The set $\pi_i^{-1}(X_i\setminus\{t\})$ is a neighborhood of $z$ in $\mathbf{X}$, so there exists $V_0\in\mathcal{V}$ such that $V_0\subseteq \pi_i^{-1}(X_i\setminus\{t\})$. Then $\pi_i(V_0)\subseteq X_i\setminus\{t\}$ and, in consequence, $t\in X_i\setminus\pi_i(V_0)\in\mathcal{E}$. This proves that $\bigcup\limits_{n\in\omega}X_n=\bigcup\mathcal{E}$; thus, $\bigcup\limits_{n\in\omega}X_n$ is a cuf set. Hence $\mathbf{CPcofinQM}$ implies $\mathbf{UT}(\aleph_0, cuf, cuf)$. We notice that if $X_n$ is countable for each $n\in\omega$, then $\mathbf{CPcofin(\omega)QM}$ implies that $\bigcup\limits_{n\in\omega}X_n$ is a cuf set. Hence $\mathbf{CPcofin(\omega)QM}$ implies $\mathbf{UT}(\aleph_0,\aleph_0, cuf)$.  To complete the proof, it suffices to apply Theorem \ref{t02.2}.
\end{proof}

The following corollary summarizes our main results on cuf products of cofinite cuf spaces.

\begin{corollary}
\label{c03.20}
 $(\mathbf{ZF})$
\begin{enumerate}
\item[(i)]  $\mathbf{UT}(cuf, cuf, cuf)$, $\mathbf{UT}(\aleph_0, cuf, cuf)$, $\mathbf{CPcofinQM}$, ``Every cuf prod\-uct of cofinite cuf spaces is quasi-metrizable'', ``Every cuf product of cofinite cuf spaces is first-countable'', ``Every cuf product of cofinite cuf spaces is of countable pseudocharacter'', ``Every cuf product of cofinite cuf spaces is of cuf pseudocharacter'' and ``Every cuf product of one-point Hausdorff compactifications of infinite discrete cuf spaces is met\-rizable'' are all equivalent statements.

\item[(ii)] $\mathbf{UT}(\aleph_0, \aleph_0, cuf)$, $\mathbf{CPcofin(\omega)QM}$, ``Every cuf product of countable co\-finite spaces is quasi-metrizable'', ``Every cuf product of countable cofinite spaces is first-countable'', ``Every cuf product of countable cofinite spaces is of countable pseudocharacter'',``Every cuf product of countable cofinite spaces is of cuf pseudocharacter"  and ``Every cuf product of one-point Hausdorff compactifications of denumerable discrete spaces is metrizable'' are all equivalent statements.

\item[(iii)] In (i) and (ii), cuf products can be replaced with countable products. 
\end{enumerate}
\end{corollary}

It was shown in \cite{HowTach} that $\mathbf{UT}(\aleph_0, cuf, cuf)$ does not imply in $\mathbf{ZF}$ that all countable products of metrizable spaces are quasi-metrizable. Hence, from Corollary \ref{c03.20}, we deduce the following theorem:

\begin{theorem}
\label{t3.20}
There is model of $\mathbf{ZF}$ in which $\mathbf{CPcofinQM}$ is true and $\mathbf{CPQM}$ is false. 
\end{theorem}

Let us give a little deeper insight into Theorems \ref{t02.13}, \ref{t02.14} and Remark \ref{r02.14} by showing the following new results:

\begin{theorem}
\label{t03.21}
$(\mathbf{ZF})$ Each of the following statements implies the one beneath it:
\begin{enumerate}
\item[(i)] $\mathbf{CMC}$.

\item[(ii)] For every family $\{\mathbf{X}_{n}: n\in \omega\}$ of
topological spaces, if there exists a family $\{\mathcal{B}_{n}:n\in \omega\}$
such that, for every $n\in \omega$, $\mathcal{B}_{n}$ is a $\sigma $%
-locally finite (resp., $\sigma$-point-finite, $\sigma$-interior-preserving) base of $\mathbf{X}_{n}$, then the product $\mathbf{X}%
=\prod\limits_{n\in \omega}\mathbf{X}_{n}$ has a $\sigma $-locally
finite (resp., $\sigma$-point-finite) base.

\item[(iii)] For every family $\{\mathbf{X}_{n}: n\in \omega\}$ of
countable, compact metrizable spaces, if there exists a family $\{\mathcal{B}_{n}:n\in \omega\}$ such that, for every $n\in \omega$, $\mathcal{B}_{n}$
is a $\sigma $-locally finite base of $\mathbf{X}_{n}$, then the product $%
\mathbf{X}=\prod\limits_{n\in \omega}\mathbf{X}_{n}$ has a $\sigma $%
-interior-preserving base.

\item[(iv)] $\mathbf{CPcofin(\omega)QM}$.

\item[(v)] $\mathbf{CMC}_{\omega }$.
\end{enumerate}
\end{theorem}

\begin{proof}
(i) $\rightarrow $ (ii) To prove the first implication, we assume $\mathbf{CMC}$ and suppose that $\{%
\mathbf{X}_{n}:n\in \omega\}$ is a family topological spaces, and
$\{\mathcal{B}_{n}:n\in \omega\}$ is a family such that, for every $%
n\in \omega$, $\mathcal{B}_{n}$ is a $\sigma $-locally finite (resp., $\sigma$-point-finite, $\sigma$-interior-preserving) base of $%
\mathbf{X}_{n}$. For every $n\in\omega$, let $\mathcal{G}_n$ be the collection of all mappings $g:\omega\to [\mathcal{P}(\mathcal{B}_n)]^{<\omega}$ such that $\mathcal{B}_n=\bigcup\limits_{m\in\omega}g(m)$ and, for every $m\in\omega$, $g(m)$ is a locally finite (resp., point-finite, interior-preserving) family.  By $\mathbf{CMC}$, we can fix a family $\{G_n: n\in\omega\}$ of non-empty finite sets such that, for every $n\in\omega$, $G_n\subseteq\mathcal{G}_n$. For $n,m\in\omega$, we define $\mathcal{B}_{n,m}=\bigcup\limits_{g\in G_n}g(m)$. Then, for arbitrary $n,m\in\omega$, $\mathcal{B}_{n,m}$ is a locally finite (resp., point-finite, interior-preserving) family in $\mathbf{X}_n$ such that $\mathcal{B}_n=\bigcup_{m\in\omega}\mathcal{B}_{n,m}$. For $i\in\omega$, let $\pi_i:\prod\limits_{n\in\omega}\mathbf{X}_n\to \mathbf{X}_i$ be the projection. For every $k\in\mathbb{N}$, we define a family $\mathcal{W}_k$ of open subsets of $\mathbf{X}$ as follows:

$$\mathcal{W}_k=\{\bigcap{\pi}_n^{-1}(B_n): n\in k\text{ and } B_n\in\bigcup\limits_{m\in k}\mathcal{B}_{n,m}\}.$$

Let $\mathcal{W}=\bigcup\limits_{k\in\mathbb{N}}\mathcal{W}_k$. Since, for every $n\in\omega$,  $\mathcal{B}_n$ is a base of $\mathbf{X}_n$, it follows easily that $\mathcal{W}$ is a base of $\mathbf{X}$. If, for arbitrary $n,m\in\omega$, the family $\mathcal{B}_{n,m}$ is interior-preserving (resp., point-finite in $\mathbf{X}_n$, then, for every $k\in\mathbb{N}$, the family $\mathcal{W}_k$ is interior-preserving (resp., point-finite) in $\mathbf{X}$. We assume that, for all $n,m\in\omega$, $\mathcal{B}_{n,m}$ is locally finite in $\mathbf{X}_n$ and prove that, for every $k\in\mathbb{N}$, the family $\mathcal{W}_k$ is locally finite in $\mathbf{X}$.

We fix $k\in\mathbb{N}$ and a point $x$ of $\mathbf{X}$. We can fix a collection $\{V_n: n\in k\}$ such that, for every $n\in k$, $V_n$ is an open set in $\mathbf{X}_n$  such that $x(n)\in V_n$ and the set $\mathcal{E}_n(x)=\{ B\in\bigcup\limits_{m\in k}\mathcal{B}_{n,m}: B\cap V_n\neq\emptyset\}$ is finite. Let $V=\bigcap\limits_{n\in k}\pi^{-1}(V_n)$. Then $V$ is open in $\mathbf{X}$, $x\in V$ and the set $\{W\in\mathcal{W}_k: V\cap W\neq\emptyset\}$ is finite.

The implication (ii)$\rightarrow$ (iii) is obvious. 

(iii)$\rightarrow $(iv) Let us consider any family $\{\mathbf{Y}_n: n\in\omega\}$ of one-point Hausdorff compactifications of denumerable discrete spaces. For every $n\in\omega$, let $\infty_n$ be the unique accumulation point of $\mathbf{Y}_n$ and put
$$\mathcal{B}_n=\{\{x\}: x\in Y_n\setminus\{\infty_n\}\}\cup\{Y_n\setminus F: F\in [\mathcal{P}(Y_n\setminus\{\infty_n\})]^{<\omega}\}$$
\noindent where $Y_n$ is the set of all points of $\mathbf{Y}_n$. Then, for every $n\in\omega$, the family $\mathbf{B}_n$ is a countable base of $\mathbf{Y}_n$. Assume that (iii) holds. Then $\mathbf{Y}=\prod\limits_{n\in\omega}$ has a $\sigma$-interior-preserving base. By Theorem \ref{t02.14} that $\mathbf{Y}$ is quasi-metrizable, so first-countable. Now, it follows from the proof to Theorem \ref{t03.12} that (iii) implies (iv).

(iv)$\rightarrow$(v) Let $\mathcal{A}=\{A_n: n\in\omega\}$ be a disjoint family of non-empty countable sets and assume that (iv) holds. It follows from Theorem \ref{t03.18}(ii) that the set  $A=\bigcup\limits_{n\in\omega}A_n$ is a cuf set. Let $\{F_n: n\in\omega\}$ be a family of finite sets such that $A=\bigcup\limits_{n\in\omega}F_n$. To show that $\mathcal{A}$  has a multiple choice function, for every $n\in\omega$, we define $m(n)=\min\{m\in\omega: F_m\cap A_n\neq\emptyset\}$ and put $f(n)=A_n\cap F_{m(n)}$. Then $f$ is a multiple choice function for $\mathcal{A}$. 
\end{proof}

It is worthwhile to compare condition (ii) of Theorem \ref{t03.21} with the following Theorems \ref{t03.22} and \ref{t03.23}.

\begin{theorem}
\label{t03.22}
 $(\mathbf{ZF})$ Let $J$ be a non-empty cuf set and let $\{\mathbf{X}_j: j\in J\}$ be a collection of $T_1$-spaces. Suppose that there exists a collection $\{\mathcal{B}_{j, m}: m\in\omega, j\in J\}$ such that, for arbitrary $j\in J$ and $m\in\omega$, $\mathcal{B}_{j, m}$ is an interior-preserving family in $\mathbf{X}_j$ such that $\mathcal{B}_j=\bigcup\limits_{m\in\omega}\mathcal{B}_{j, m}$ is a base of $\mathbf{X}_j$. Then $\mathbf{X}=\prod\limits_{j\in J}\mathbf{X}_j$ is non-Archimedeanly quasi-metrizable.
\end{theorem}

\begin{proof}
Let $\{K_n; n\in\omega\}$ be a collection of non-empty finite sets such that $K_n\subseteq K_{n+1}$ for each $n\in\omega$ and $J=\bigcup\limits_{n\in\omega}K_n$. Let $\mathbf{X}_j=\langle X_j, \tau_j\rangle$ for each $j\in\omega$ and let $X=\prod\limits_{j\in J}X_j$. The standard proof to Theorem \ref{t02.14} (see the proof to Theorem 7.1 in \cite{fl} or Theorem 10.2 on p. 489 in \cite{KV}) shows that there exists a family $\{d_j: j\in J\}$ such that, for every $j\in J$, $d_j$ is a non-Archimedean quasi-metric on $X_j$ such that $\tau_j=\tau(d_j)$ for each $j\in J$. For $n\in\omega$ and $x,y\in X$, let $\rho_n(x,y)=\max\{\min\{d_j(x(j), y(j)), 1\}: j\in K_n\}$  and let 
$$\rho(x, y)=\sup\{\frac{\rho_n(x, y)}{2^{n+1}}: n\in\omega\}.$$
It is easy to check that $\rho$ is a non-Archimedean quasi-metric on $X$ which induces the topology of $\mathbf{X}$ (see the proof to Theorem 2.1 in \cite{ew}). 
\end{proof}

\begin{theorem}
\label{t03.23}
 $(\mathbf{ZF})$ Let $J$ be a non-empty cuf set and let $\{\mathbf{X}_j: j\in J\}$ be a collection of $T_3$-spaces. Suppose that there exists a collection $\{\mathcal{B}_{j, m}: m\in\omega, j\in J\}$ such that, for arbitrary $j\in J$ and $m\in\omega$, $\mathcal{B}_{j, m}$ is a locally finite family in $\mathbf{X}_j$ such that $\mathcal{B}_j=\bigcup\limits_{m\in\omega}\mathcal{B}_{j, m}$ is a base of $\mathbf{X}_j$. Then $\mathbf{X}=\prod\limits_{j\in J}\mathbf{X}_j$ is metrizable.
\end{theorem}

\begin{proof} Using the same methods as, for instance, in  \cite{cr}, one can show that there exists a collection $\{\rho_j: j\in J\}$ such that, for every $j\in J$, $\rho_j$ is a metric on $X_j$ which induces the topology of $\mathbf{X}_j$. To conclude the proof, it suffices to apply Theorem \ref{t02.6}.
\end{proof}

\begin{remark}
\label{r03.24}
It is a well-known fact that $\mathbf{DC}$ implies $\mathbf{CAC}$ in $\mathbf{ZF}$ (see, e.g., p. 324  in \cite{hr}). Therefore, $\mathbf{CMC}$ is satisfied in every model of $\mathbf{ZF}+\mathbf{DC}$. Let $\mathcal{M}$ be any model of $\mathbf{ZF}+\mathbf{DC}$  such that $\mathbf{MP}$ fails in $\mathcal{M}$. As we have already mentioned in Remark \ref{r02.14}, an example of such a model $\mathcal{M}$ was constructed in \cite{gtw}. It follows from Theorem \ref{t03.21} that  condition (ii) of Theorem \ref{t03.21} is satisfied in $\mathcal{M}$. In view of Theorem  8 of \cite{hkrs}, it holds in $\mathbf{ZF}$ that every non-paracompact metrizable space admits no $\sigma$-point-finite base and, therefore, by Theorem \ref{t02.14}, is not non-Archimedeanly quasi-metrizable. 
\end{remark}

\begin{corollary} 
\label{c03.25}
In $\mathbf{ZF}+\mathbf{DC}$, condition (ii) of Theorem \ref{t03.21} implies neither $\mathbf{MP}$ nor the sentence: Every metrizable space is non-Archimedeanly quasi-metrizable.
\end{corollary}

Taking the opportunity, let us show that the spaces $\omega_1+1$ and $\omega_1$, both equipped with their order topology, cannot be quasi-metrizable in $\mathbf{ZF}$. It was shown in Lemmas 4.1 and 4.2  in \cite{gt} that the spaces $\omega_1+1$ and $\omega_1$ are not metrizable in $\mathbf{ZF}$. 

\begin{proposition}
\label{p03.25}
$(\mathbf{ZF})$ The spaces $\omega_1+1$ and $\omega_1$ equipped with their order topology are not quasi-metrizable. 
\end{proposition}

\begin{proof}
It is sufficient to mimic the proofs to Lemmas 4.1 and 4.2 in \cite{gt} and replace in their proofs the following:
\begin{enumerate}
\item[(1)] the metric $d$ with a quasi-metric $d$ inducing the order topology of $\omega_1+1$ (resp., the order topology of $\omega_1$),
\item[(2)] numbers $d(a, f(s))$ with $d(f(s), a)$.
\end{enumerate}
\end{proof}

\subsection{Countable products of two-point Hausdorff compactifications of denumerable discrete spaces}

Let us come back for a while to countable products of one-point Hausdorff compactifications of denumerable discrete spaces to show somewhat later important differences between them and countable products of two-point Hausdorff compactifications of denumerable discrete spaces.

 The following theorem follows directly from Theorem \ref{t03.10}:

\begin{theorem}
\label{t03.26}
$(\mathbf{ZF})$ A countable product of one-point Hausdorff compactifications of denumerable discrete spaces is first-countable iff it is (quasi)-metrizable
\end{theorem}

We are going to show that one-point Hausdorff compactifications cannot be replaced with two-point Hausdorff compactifications of denumerable discrete spaces in Theorem \ref{t03.26}. To do this, we need the following theorem: 

\begin{theorem}
\label{t03.27}
Let $\mathcal{M}$ be a model of $\mathbf{ZF}$ in which there exists a family $\{\mathbf{X}_n: n\in\omega\}$ of non-empty metrizable spaces  such that the family $\{X_n: n\in\omega\}$ of their underlying sets is disjoint and it does not have a partial choice function, and the direct sum $\bigoplus\limits_{n\in\omega}\mathbf{X}_n$ is not metrizable. Then there exists in $\mathcal{M}$ a family $\{\mathbf{Y}_n: n\in\omega\}$ of (quasi)- metrizable spaces such that the product $\mathbf{Y}=\prod\limits_{n\in\omega}\mathbf{Y}_n$ is first-countable and not (quasi)-metrizable. 
\end{theorem}
\begin{proof}
In what follows, we work inside the model $\mathcal{M}$. We fix a family of (quasi)-metrizable spaces  $\{\mathbf{X}_n: n\in\omega\}$ satisfying the assumptions of our theorem. Let $\infty$ be an element of $\mathcal{M}$ such that $\infty\notin\bigcup\limits_{n\in\omega}X_n$. For every $n\in\omega$, let ${Y}_n={X}_n\cup\{\infty\}$ and $\mathbf{Y}_n=\mathbf{X}_n\oplus\{\infty\}$, that is, $\mathbf{Y}_n$ is the direct sum of the space $\mathbf{X}_n$ and the one-point discrete space $\{\infty\}$. Let $y\in\prod\limits_{n\in\omega}Y_n$. We notice that, since $\{X_n: n\in\omega\}$ does not have a partial choice function, there exists $n_y\in\omega$ such that $y(n)=\infty$ for every $n\in\omega\setminus n_y$. For every $i\in n_y$,  let $\mathcal{B}_i(y(i))$  be a countable base of neighborhoods of $y(i)$ in $\mathbf{Y}_i$. Let $\mathcal{B}(y)$ be the collection of all sets of the form $\prod\limits_{n\in\omega}V_n$ where $V_n\in\mathcal{B}(y(n))$ for every $n\in n_y$, and $V_n=Y_n$ for every $n\in\omega\setminus n_y$. Then $\mathcal{B}(y)$ is a countable base of neighborhoods of $y$ in $\mathbf{Y}$. Since $\bigoplus\limits_{n\in\omega}\mathbf{X}_n$ is assumed to be not (quasi)-metrizable, the space $\mathbf{Y}$ is also not (quasi)-metrizable. 
\end{proof}

In our next theorem, we involve the forms $\mathbf{PC}(\aleph_0, 2, \infty)$ and $\mathbf{CMC}_{\omega}$ defined in Definition \ref{d1.10} (see also Remark \ref{r1.11}(iv)).

\begin{theorem}
\label{t03.28} 
\begin{enumerate}
\item[(a)] $(\mathbf{ZF})$ The disjunction $\mathbf{PC}(\aleph_0, 2, \infty)\vee\mathbf{CMC}_{\omega}$ follows from  the following statement: 

$\mathbf{CP2pHC}$: For every family $\{\mathbf{Y}_n: n\in\omega\}$ of two-point Hausdorff compactifications of denumerable discrete spaces, the product $\prod\limits_{n\in\omega}\mathbf{Y}_n$ is (quasi)-metrizable iff it is first-countable. 
\item[(b)] There exists a model $\mathcal{M}$ of $\mathbf{ZF}$ in which $\mathbf{CP2HC}$ is false.
\end{enumerate}
\end{theorem}
\begin{proof}
($a$) Suppose that $\mathcal{M}$ is a model of $\mathbf{ZF}$ in which $\mathbf{PC}(\aleph_0, 2, \infty)$ and $\mathbf{CMC}_{\omega}$ are both false. In the sequel, let us work inside $\mathcal{M}$. We fix the following families in $\mathcal{M}$. 
\begin{enumerate}
\item[(i)] a family $\mathcal{C}=\{C_n: n\in\omega\}$ of two-element sets such that $\mathcal{C}$ fails to have a partial choice function in $\mathcal{M}$;
\item[(ii)] a family $\mathcal{A}=\{A_n: n\in\omega\}$ of denumerable sets which does not have a partial multiple choice function in $\mathcal{M}$.
\end{enumerate} 
 For every $n\in\omega$, let $X_n=(A_n\times C_n)\cup C_n$. We may assume that, for every pair $n,m$ of distinct elements of $\omega$, $X_n\cap X_m=\emptyset$ and  $C_n\cap (A_n\times C_n)=\emptyset$. Now, we describe how to choose in $\mathbf{ZF}$ a sequence $(\tau_n)_{n\in\omega}$ such that, for every $n\in\omega$, $\mathbf{X}_n= \langle X_n, \tau_n\rangle$ is a two-point Hausdorff compactification of the denumerable discrete space $\mathbf{Z}_n=\langle A_n\times C_n, \mathcal{P}(A_n\times C_n)\rangle$. 
For a fixed $n\in\omega$ and every $x\in X_n$, we define a neighborhood base $\mathcal{B}_n(x)$ of $x$ in $\mathbf{X}_n$ as follows:
$\mathcal{B}(x)=\{\{x\}\}$ if $x\in A_n\times C_n$; for $c\in C_n$, we put $\mathcal{B}_n(c)=\{U\subseteq (A_n\times\{c\})\cup\{c\}: c\in U \text{ and } (A_n\times \{ c\})\setminus U\in [\mathcal{P}(A_n\times\{ c\})]^{<\omega}\}$.

Clearly, for every $n\in\omega$, the space $\mathbf{X}_n$ is metrizable. Suppose that the space $\mathbf{X}=\bigoplus\limits_{n\in\omega}\mathbf{X}_n$ is quasi-metrizable. Let $d$ be a quasi-metric which induces the topology of $\mathbf{X}$. For every $n\in\omega$ and every pair $x,y$ of points of $X_n$, let $d_n(x, y)=\min\{d(x,y), 1\}$. Then, for every $n\in \omega$,  $d_n$ is a quasi-metric on $X_n$ such that $\tau(d_n)=\tau_n$. For each $i\in\mathbb{N}$, let $V_i=\bigcup\limits_{n\in\omega}\bigcup\limits_{c\in C_n}B_{d_n}(c, \frac{1}{i})$. Given $n\in\omega$, the set $E_n=\{i\in\mathbb{N}: (A_n\times C_n)\setminus V_i\neq\emptyset\}$ is non-empty, so we can define $i(n)=\min E_n$ and $F_n=(A_n\times C_n)\setminus V_{i(n)}$. Then $\{F_n: n\in\omega\}$ is a collection of non-empty finite sets such that, for every $n\in\omega$, $F_n\subseteq A_n\times C_n$. This contradicts the assumption that $\mathcal{A}$ does not have a multiple choice function. Hence $\mathbf{X}$ is not quasi-metrizable. As in the proof to Theorem \ref{t03.27}, Let $\infty$ be an element of $\mathcal{M}$ such that $\infty\notin\bigcup\limits_{n\in\omega}X_n$ and, for every $n\in\omega$, let $\mathbf{Y}_n=\mathbf{X}_n\oplus\{\infty\}$. Then, for every $n\in\omega$, the space $\mathbf{Y}_n$ is a two-point Hausdorff compactification of a denumerable discrete space. The proof to Theorem \ref{t03.27} shows that $\prod\limits_{n\in\omega}\mathbf{Y}_n$ is first-countable but not quasi-metrizable. This completes the proof to ($a$).

To prove ($b$), let us notice that it was shown in the proof to Theorem 4 of \cite{t} that $\mathbf{PC}(\aleph_0,2, \infty)$ is false in the Feferman-Blass model $\mathcal{M}$15 of \cite{hr}. It can be deduced from the proof to Theorem 4 in \cite{t} that $\mathbf{CMC}_{\omega}$ is also false in $\mathcal{M}$15. Hence, it follows from ($a$) that  $\mathcal{M}$15 is an example of a model of $\mathbf{ZF}$ which $\mathbf{CP2HC}$ fails.

To explain that $\mathbf{CMC}_{\omega}$ is false in $\mathcal{M}$15, let us use the same notation as in the proof to Theorem 4 in \cite{t}. In the proof to Theorem 4 in \cite{t}, the family $\mathcal{R}=\{\{\delta(a_{n}),\delta(\omega\setminus a_{n})\}: n\in\omega\}$ of pairs is shown to have no partial choice function in $\mathcal{M}$15. In what follows, let us work inside $\mathcal{M}$15.  Since, for every $n\in\omega$, the sets  $\delta(a_{n})$ and $\delta(\omega\setminus a_n)$ are both countable, it follows that each one of the sets $A_n=\delta(a_n)\cup\delta(\omega\setminus a_n)$ is also countable. Since $A_{n}\subseteq\mathcal{P}(\omega)$ for every $n\in\omega$, we may identify $\{A_n: n\in\omega\}$ with a family of subsets of $\mathbb{R}$. Moreover, for every $n\in\omega$, $\delta(a_n)\cap\delta(\omega\setminus a_n)=\emptyset$. All this taken together implies that if the family $\{A_n: n\in\omega\}$ had a partial multiple choice function in $\mathcal{M}$15, the family $\mathcal{R}$ would have a partial choice function in $\mathcal{M}$15. Indeed, suppose that $J$ is an infinite subset of $\omega$ and $\{F_n: n\in J\}$ is a family of non-empty finite sets such that $F_n\subseteq A_n$ for each $n\in J$. Let $n\in J$. Identifying $A_n$ with a subset of $\mathbb{R}$, we can define $x_n$ as the smallest (with respect to the standard linear order of $\mathbb{R}$) element of $F_n$. Then, we can define a choice function $\psi$ of $\{\{\delta(a_n), \delta(\omega\setminus a_n)\}: n\in J\}$ as follows: $\psi(n)=\delta(a_n)$ if $x_n\in\delta(a_n)$, and $\psi(n)=\delta(\omega\setminus a_n)$ if $x_n\in\delta(\omega\setminus a_n)$. Hence $\mathbf{CMC}_{\omega}$ fails in $\mathcal{M}$15. 
\end{proof}

\begin{remark}
\label{r03.28}
An alternative proof to ($b$) of Theorem \ref{t03.28} is to show a permutation  model $\mathcal{N}$ of $\mathbf{ZFA}$ in which $\mathbf{PC}(\aleph_0,2, \infty)$ and $\mathbf{CMC}_{\omega}$ are both false and, next, using the theorems of Note 103 of \cite{hr}, to transfer the sentence $\neg\mathbf{PC}(\aleph_0, 2\infty)\wedge\neg\mathbf{CMC}_{\omega}$ from $\mathcal{N}$ into a model of $\mathbf{ZF}$. 

To show such a permutation model $\mathcal{N}$, similarly to Part III (2) in \cite{hr}, let us specify the set $A$ of atoms, the group $\mathcal{G}$ of permutations of $A$ and the ideal $\mathcal{S}$ of supports which is an ideal of subsets of $A$. More details about permutation models can be found in \cite{j}.

The set $A$ of atoms in the ground model is denumerable and $A=\bigcup\limits_{i\in\omega}P_i$ where, for every pair of distinct $i,j\in\omega$, $P_i\cap P_j=\emptyset$ and $P_i$ is denumerable. For each $i\in\omega$, let $D_{i}=\{P_{2i},P_{2i+1}\}$. The group $\mathcal{G}$ of permutations of $A$ is the set of all
permutations $\pi $ of $A$ such that, for every $i\in \omega ,$ either $\pi $
maps $P_{2i}$ onto $P_{2i}$ and $P_{2i+1}$ onto $P_{2i+1}$ or, it maps $P_{2i}$ onto $P_{2i+1}$
and $P_{2i+1}$ onto $P_{2i}$. In any case, for every $i\in \omega$, $\pi
(D_{i})=\{\pi (P_{2i}),\pi (P_{2i+1})\}=D_{i}$. The ideal of supports $%
\mathcal{S}$ is the set of all subsets $Q$ of $A$ such that there exists $i_Q\in\omega$ such that
$Q\subseteq \bigcup\limits_{i\in 2i_Q+2} P_i$. Let $%
\mathcal{N}$ be the permutation model determined by $\mathcal{G}$ and $\mathcal{%
S}$. It is a routine procedure to verify that, in $\mathcal{N}$, the family $\{D_i: i\in\omega\}$ does not have a partial choice function, and the family $\{\bigcup D_i; i\in\omega\}$ does not have a partial multiple choice function.

Using the theorems of Note 103 of \cite{hr}, one can transfer the sentence $\neg\mathbf{PC}(\aleph_0, 2\infty)\wedge\neg\mathbf{CMC}_{\omega}$ from $\mathcal{N}$ into a model of $\mathbf{ZF}$.
\end{remark}

\begin{remark}
\label{r03.29}
Let $\mathcal{M}$ be a model of $\mathbf{ZF}$ in which $\mathbf{PC}(\aleph_0,2, \infty)$ and $\mathbf{CMC}_{\omega}$ are both false. In view of the proof to ($b$) of Theorem \ref{t03.28}, the Feferman-Blass model $\mathcal{M}$15 can be taken as $\mathcal{M}$. Then we can infer from the proof to ($a$) of Theorem \ref{t03.28} that there exists in $\mathcal{M}$ a family $\{\mathbf{X}_n: n\in\omega\}$ which satisfies the assumptions of Theorem \ref{t03.27}.  
\end{remark}

\begin{theorem}
\label{t03.30}
$(\mathbf{ZF})$  $\mathbf{CAC}_{fin}$ follows from the sentence: For every family $\{\mathbf{X}_n: n\in\omega\}$ of one-point Hausdorff compactifications of denumerable discrete spaces, if the product $\mathbf{X}=\prod\limits_{n\in\omega}\mathbf{X}_n$ is quasi-metrizable, then it is compact. 
\end{theorem}
\begin{proof} Let $\mathcal{A}=\{A_n: n\in\omega\}$ be a family of non-empty finite sets. For every $n\in\omega$, let $Y_n=(A_n\times\{0\})\cup (\omega\times\{1\})$. Let $\infty$ be an element such that $\infty\notin\bigcup\limits_{n\in\omega}Y_n$. For every $n\in\omega$, let $\mathbf{X}_n=\mathbf{Y}_n(\infty)$. Then the product $\mathbf{X}=\prod_{n\in\omega}\mathbf{X}_n$ is quasi-metrizable by Proposition \ref{p03.6} and Theorem \ref{t02.6}. Suppose that $\mathbf{X}$ is compact and consider the family $\mathcal{F}=\{\pi_n^{-1}(A_n\times\{0\}): n\in\omega\}$ where, for every $n\in\omega$, $\pi_n: \mathbf{X}\to\mathbf{X}_n$ is the projection. Then the family $\mathcal{F}$ consists of closed sets of $\mathbf{X}$ and has the finite intersection property. By the compactness of $\mathbf{X}$, there exists $f\in\bigcap\limits_{n\in\omega}\pi_n^{-1}(A_n\times\{0\})$. By defining $h(n)=\pi_n(f(n))$ for every $n\in\omega$, we obtain a choice function $h$ of the family $\{A_n\times\{0\}: n\in\omega\}$. Hence $\mathcal{A}$ has a choice function.
\end{proof}

\section{Open problems and comments}
\label{s4}

\subsection{The list of open problems}
\label{s4.1}
For the convenience of readers and researchers, we list below several intriguing open problems.
\begin{question}
\label{q4.1}
Does $\mathbf{CPQM}$ imply $\mathbf{CPM}_{le}$ in $\mathbf{ZF}$?
\end{question}
\begin{question}
\label{q4.2}
Does $\mathbf{CPQM}$ imply $\mathbf{CPQMBi}$ in $\mathbf{ZF}$? 
\end{question}
\begin{question}
\label{q4.3}
Does $\mathbf{CPQMBi}$ imply $\mathbf{CPQM}$ in $\mathbf{ZF}$?
\end{question}
\begin{question}
\label{q4.4}
Does $\mathbf{CPM}_{le}(C, M)$ imply in $\mathbf{ZF}$ that all countable products of compact Hausdorff quasi-metrizable spaces are quasi-metrizable?
\end{question}
\begin{question}
\label{q4.5}
Is it true in $\mathbf{ZF}$ that every compact Hausdorff quasi-metrizable space is metrizable?
\end{question}
\subsection{Comments on the open problems}
\label{s4.2}

We notice that if the answer to Question \ref{q4.5} is in the affirmative, then so is the answer to Question \ref{q4.4}. 
Regarding Question \ref{q4.5}, it is well known that every compact Hausdorff quasi-metrizable space is  metrizable in $\mathbf{ZFC}$ (see, e.g., Corollary 2.30 in \cite{fl}). Let us give a deeper but still a partial answer to Question \ref{q4.5}.

\begin{theorem}
$(\mathbf{ZF})$
\begin{enumerate}  
\item[($a$)] For every compact Hausdorff, quasi metric space $\mathbf{X}=\langle X,d\rangle$
the following are equivalent:
\begin{enumerate}
\item[(i)] $\mathbf{X}$ is Loeb;
\item[(ii)] $\langle X, d^{-1}\rangle$ is separable;
\item[(iii)] $\mathbf{X}$ and $\langle X,d^{-1}\rangle$ are both separable;
\item[(iv)] $\mathbf{X}$ is second-countable.
\end{enumerate}
 In particular, every compact, Hausdorff, quasi-metrizable Loeb space is metrizable.
\item[($b$)] $\mathbf{CAC}$ implies that every compact, Hausdorff quasi-metrizable space is metrizable.
\end{enumerate}
\end{theorem}

\begin{proof}
($a$) Let $\mathbf{X}=\langle X,d\rangle$ be a non-empty, compact, Hausdorff quasi-metric space. It follows from Proposition 2.1.11 of \cite{kunzi} that $\tau(d)\subseteq \tau(d^{-1})$. 

 (i) $\rightarrow $ (ii) Suppose that $\mathbf{X}$ has a Loeb function $f$ (see Definition \ref{d4.5}). We are going to
construct, via an easy induction, a dense set $D=\{x_{i}:i\in \omega \}$ of $%
\langle X,d^{-1} \rangle$. 

Fo a given  $n\in \omega$, we define inductively a finite subset $D_{n}$
of $X$ such that $X=\bigcup\limits_{z\in D_n}B_d(z, \frac{1}{2^n})$ as follows.  Let $x_{n,0}=f(X)$ and let $k\in\omega$ be such that the set $\{x_{n,i}; i\in k+1\}$ has been defined. If $X=\bigcup\limits_{i\in k+1}B_d(x_{n,i}, \frac{1}{2^n})$, we put $D_n=\{x_{n,i}: i\in k+1\}$. Otherwise, if $X\setminus \bigcup\limits_{i\in k+1}B_d(x_{n,i}, \frac{1}{2^n})\neq\emptyset$, we put 
$$x_{n,k+1}= f(X\setminus \bigcup\limits_{i\in k+1}B_d(x_{n,i}, \frac{1}{2^n})).$$
It follows from the compactness of $\mathbf{X}$ that there exists $k_n\in\omega$ such that 
$$X=\bigcup\limits_{i\in k_n+1}B_d(x_{n,i}, \frac{1}{2^n}).$$
We define $D_n=\{x_{n, i}: i\in k_n+1\}$. The set $D=\bigcup\limits_{n\in\omega}D_n$ is countable as a countable union of well-ordered finite sets. 

Consider any $x\in X$ and $n\in\omega$. There exists $i\in k_{n}+1$ such that $x\in B_d(x_{n,i}, \frac{1}{2^n})$. Then $x_{n,i}\in B_{d^{-1}}(x,\frac{1}{2^n})$. This implies that $D$ is dense in $\langle X,d^{-1}\rangle$. Hence (i) implies (ii). That (ii) implies (iii) follows from the inclusion $\tau(d)\subseteq \tau(d^{-1})$.

(iii) $\rightarrow $ (iv) Suppose that $D=\{x_{i}:i\in \omega \}$ is a common dense
subset of $\mathbf{X}$ and $\langle X,d^{-1}\rangle$. Clearly, 
\[
\mathcal{B}=\{B_{d}(x_{i},\frac{1}{2^n}):i,n\in \omega\}
\]%
is a countable subset of $\tau(d)$. We show that $\mathcal{B}$ is a base of $\mathbf{X}$. To see this, fix $U\in \tau(d)$ and let $x\in U$. Let $n_0\in \omega$ be such that $B_{d}(x,\frac{1}{2^{n_0}})\subseteq U$. Since, $\tau(d)\subseteq \tau(d^{-1})$, there exists $m_0\in\omega$ such that $n_0<m_0$ and $B_{d^{-1}}(x,\frac{1}{2^{m_0}})\subseteq B_{d}(x,\frac{1}{2^{n_0}})$. There exists $i_0\in\omega$ such that $x_{i_0}\in B_{d^{-1}}(x, \frac{1}{2^{m_0}})$. Then $x\in B_d(x_{i_0}, \frac{1}{2^{m_0}})\in\mathcal{B}$. If $z\in B_d(x_{i_0}, \frac{1}{2^{m_0}})$, then 

\[
d(x,z)\leq d(x,x_{i_0})+d(x_{i_0},z)<\frac{1}{2^{m_0-1}}\leq \frac{1}{2^{n_0}}.
\]%
Therefore, $ B_d(x_{i_0}, \frac{1}{2^{m_0}})\subseteq U$. In consequence, (ii) implies (iii). \smallskip

If (iv) holds, then $\mathbf{X}$ is metrizable by Theorem \ref{t02.12}. It follows from the proof to Theorem 2.1 in \cite{kert} that every second-countable compact metrizable space is Loeb. Hence (iv) implies (i).\smallskip

($b$) Now, we assume $\mathbf{CAC}$ and,  for every $n\in \omega$, we define 
\[
A_{n}=\{F\in \lbrack X]^{<\omega }:X=\bigcup\limits_{x\in F} B_{d}(x,\frac{1}{2^n})\}.
\]%
Since $\mathbf{X}$ is compact, it follows that $A_{n}\neq \emptyset $. By $\mathbf{CAC}$, we can fix a choice function $f$ of the family $\mathcal{A}%
=\{A_{n}:n\in \omega\}$. By $\mathbf{CAC}$ again, the set $D=\bigcup\limits_{n\in\omega}f(n)$ is a countable subset of $\mathbf{X}$.   Mimicking the proof of (i) $\rightarrow $ (ii) of part ($a$), one can check that $D$ is dense in $\langle X, d^{-1}\rangle$. In view of ($a$), $\mathbf{X}$ is metrizable. 
\end{proof}

\subsection*{Acknowledgements}
The authors are deeply grateful to Professor Hans-Peter K\"unzi for pointing at \cite{kunzi}.

This research of the first author did not receive any specific grant from funding agencies in the public, commercial or not-for-profit sectors. The second author thanks for the financial support of the Polish Ministry of Science and Higher Education, No. 59/20/B.

The authors declare that they have no conflict of interest.

\end{document}